\documentclass[12pt]{article}
\usepackage{enumerate,ulem,setspace}
\usepackage{amsmath,amssymb,amsthm, mathrsfs, mathtools}
\usepackage{latexsym}
\usepackage{graphics,graphicx}
\usepackage{hyperref}
\usepackage[bf, small]{titlesec}
\usepackage{authblk}
\usepackage{soul}
\usepackage{kotex}
\usepackage{lineno}
\usepackage{float} 
\usepackage[table]{xcolor}
\usepackage{colortbl}

\numberwithin{equation}{section}
\theoremstyle{plain}
\newtheorem{Thm}{Theorem}[section]

\newtheorem{Lem}[Thm]{Lemma}
\newtheorem{Fact}[Thm]{Fact}

\newtheorem{Cor}[Thm]{Corollary}

\newtheorem{Conj}[Thm]{Conjecture}
\newtheorem{Clm}[Thm]{Claim}

\theoremstyle{definition}

\newtheorem{Rmk}[Thm]{Remark}

\usepackage{standalone}
\usepackage{tikz}
\usetikzlibrary{arrows,positioning,matrix,fit,backgrounds,shapes,shapes.geometric, intersections}
\tikzstyle{vertex}=[circle, draw, inner sep=0pt, minimum size=6pt] 
\usetikzlibrary{shapes.callouts,decorations.pathmorphing, calc}
\usetikzlibrary{decorations.markings}
\tikzstyle{vertex}=[circle, draw, inner sep=0pt, minimum size=6pt]

\usetikzlibrary{arrows,matrix} 
\definecolor{LemonChiffon}{rgb}{100, 98, 80}
\definecolor{myblue}{rgb}{0,0.4,0.8}
\definecolor{orange}{rgb}{1, 0.4, 0}
\definecolor{mygreen}{rgb}{0, 0.8, 0.4}
\definecolor{myred}{rgb}{204, 0, 0}
\definecolor{violet}{RGB}{0.4,0.2,1}
\definecolor{brown}{rgb}{0.6, 0.4, 0}

\title{$2$-limited broadcast domination in cubic graphs}
 
\author[1]{\small Myungho Choi}
\author[1]{\small Boram Park\thanks{Boram Park was supported by the National Research Foundation of Korea(NRF) grant funded by the Korea government(MSIT) (No. RS-2025-00523206), and supported by the New Faculty Startup Fund from Seoul National University.
}}

\affil[1]{\footnotesize Department of Mathematics Education, Seoul National University, Seoul 08826} 

\date{}

\begin{document}
\maketitle
\begin{abstract}  
For a graph $G$, a function $f:V(G) \to \{0,1,2\}$ is called a $2$-limited dominating broadcast on $G$ if for every vertex $u$, there exists a vertex $v$ such that $f(v)>0$ and the distance between $u$ and $v$ in $G$ is at most $f(v)$.
The {\it cost} of $f$ means the value $\sum_{v\in V(G)}f(v)$, and the {\it $2$-limited broadcast domination number} of $G$, denoted by $\gamma_{b,2}(G)$, is the cost of a $2$-limited dominating broadcast on $G$ with minimum cost.
Henning, MacGillivray, and Yang (2020) conjectured that $\gamma_{b,2}(G)\leq \frac{|V(G)|}{3}$ for every cubic graph $G$. 
In this paper, we confirm the conjecture.
 \end{abstract}


\section{Introduction} 
Let $G$ be a finite simple graph with no loops.
The set of nonnegative integers is denoted by $\mathbb{Z}_{\geq0}$.
Let $f:V(G)\rightarrow \mathbb{Z}_{\geq0}$ be a function. The {\it cost} of $f$ is defined by \rm{cost}$(f) =\sum_{v\in V(G)}f(v)$.
For $i\in \mathbb{Z}_{\geq0}$, we let
$V_f^+=\{u \in V(G) \mid f(u)>0\}$ and $V_f^i=\{u \in V(G) \mid f(u)=i\}$.
We say that a vertex $u$ {\it hears from a vertex $v$ in $f$} when $v \in V_f^+$ and 
the distance between two vertices $u$ and $v$ is at most $f(v)$.
We call a function $f:V(G)\rightarrow \mathbb{Z}_{\geq0}$ a {\it  dominating broadcast} in $G$ if for each vertex $u\in V(G)$,
there exists a vertex $v\in V_f^+$ from which $u$ hears. 
The concept of dominating broadcast  on a graph was introduced by Erwin in 2001 \cite{Erwin1, Erwin2}. After its introduction, it was studied subsequently in various directions, see \cite{BC3,BC6,HMY2018,BC7,RK,SMW2022,SMW2023} for some related work and see \cite{HMY2021} for a survey.

Let $k$ be a positive integer. For a dominating broadcast $f$ on a graph $G$, if $f(v)\le k$ for every $v\in V(G)$, then $f$ is called a $k$-\textit{limited dominating broadcast}, abbreviated as a $k$-LD broadcast, on $G$. A $k$-limited dominating broadcast with minimum cost is called a \textit{minimum $k$-limited dominating broadcast}. The $k$-\textit{limited broadcast domination number} $\gamma_{b,k}(G)$ of $G$ is the cost of a minimum $k$-limited dominating broadcast on $G$.  The problem of finding a minimum $k$-limited dominating broadcast  is NP-complete \cite{CHMPP2018}. For $k=2$,
it was shown in \cite{CHMPP} that $\gamma_{b,2}(G)\le \left\lceil \frac{4|V(G)|}{9} \right\rceil$ for a connected graph $G$, and the bound is tight. In \cite{HMY2020}, Henning, MacGillivray, and Yang  posed a conjecture with a stronger bound for cubic graphs.  

\begin{Conj}[\cite{HMY2020}]\label{conj} For a cubic graph $G$, $\gamma_{b,2}(G)\le \frac{|V(G)|}{3}$.
\end{Conj}
Note that the coefficient $\frac{1}{3}$ of $|V(G)|$ in the bound is best possible, since $\gamma_{b,2}(K_{3,3})=2$. In \cite{HMY2020}, it was noted that Conjecture~\ref{conj} was motivated by Reed's $\frac{1}{3}$-conjecture (1996) on the domination number of a connected cubic graph. Reed's $\frac{1}{3}$-conjecture~\cite{Reed1996} states that a connected cubic graph with $n$ vertices has domination number at most $\left\lceil\frac{n}{3}\right\rceil$, and infinitely many counterexamples to Reed's $\frac{1}{3}$-conjecture have been found by Kostochka and Stodolsky \cite{KS-D} and Kelmans \cite{Kel}. It was observed that those known counterexamples $G$ to Reed's $\frac{1}{3}$-conjecture have $2$-limited broadcast domination number at most $\frac{|V(G)|}{3}$, and so Conjecture~\ref{conj} was posed by Henning, MacGillivray, Yang \cite{HMY2020}. They confirmed the conjecture for every cubic graph having neither $C_4$ nor $C_6$ as an induced subgraph.

\begin{Thm}[\cite{HMY2020}]\label{thm:cubic:C4C6}
For every cubic graph $G$ without induced $4$- and  $6$-cycles, $\gamma_{b,2}(G)\le \frac{|V(G)|}{3}$. 
\end{Thm}
 
Later, Park \cite{park20232} improved the result as follows.
 
\begin{Thm}[\cite{park20232}]\label{thm:cubic:C4}
For every cubic graph $G$ without induced $4$-cycles, $\gamma_{b,2}(G)\le \frac{|V(G)|}{3}$. 
\end{Thm}

Our main result is to prove Conjecture~\ref{conj}. More precisely, we prove the following. Let $n_i(G)$ be the number of vertices with degree $i$ in a graph $G$ and $b(G)$ be the number of bad connected components of $G$, where a \textit{bad connected component}  means a graph isomorphic to  an induced $4$-cycle or one subdivision $K^*_4$ of $K_4$ (see Figure~\ref{fig:A bad graph}). A \textit{subcubic graph} is a graph in which every vertex has degree at most $3$.

\begin{figure}[h!]\centering
\includegraphics[page=1]{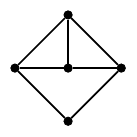}
\caption{The graph $K^*_4$} \label{fig:A bad graph}
\end{figure}

\begin{Thm}\label{Thm:Main_subcubic}
For a subcubic graph $G$, $9\gamma_{b,2}(G)\leq 9n_0(G)+5n_1(G)+4n_2(G)+3n_3(G)+2b(G)$.
\end{Thm}

As a corollary of Theorem~\ref{Thm:Main_subcubic}, Conjecture~\ref{conj} holds.

\begin{Thm}\label{Thm:Main_cubic}
Conjecture~\ref{conj} is true.
\end{Thm}

In this paper, we  prove Theorem~\ref{Thm:Main_cubic}.  Section 2 provides straightforward observations and Section 3 gives the proof of Theorem~\ref{Thm:Main_subcubic}.

Here are some notations and terminology used in the paper. Let $G$ be a graph. If a vertex $v$ has degree $d$, we say that $v$ is a {\it $d$-vertex}. A {\it $d^-$-vertex}  (resp. {\it $d^+$-vertex}) is a vertex of degree at most $d$ (resp. at least $d$). 
If a $d$-vertex (resp. $d^+$-vertex, $d^-$-vertex) $u$ is a neighbor of $v$, then $u$ is called a $d$-neighbor (resp. $d^+$-neighbor, $d^-$-neighbor) of $v$.
For a $k$-vertex $v$, if the neighbors of $v$ are $d_1$-,\ldots, $d_k$-vertex, then we call a vertex $v$ a {\it $(d_1, \ldots, d_k)$-vertex}. For example, a $(2^-,2,3)$-vertex means a $3$-vertex with one $2^-$-neighbor, one $2$-neighbor, and one $3$-vertex.
For a vertex $v$,
the degree of $v$  and the neighborhood of $v$ are denoted by 
$d_G(v)$ and $N_G(v)$,  respectively, and we denote by $N_G[v]=N_G(v)\cup \{v\}$.
For a subset $X \subset V(G)$,
$N_G(X)=({ \cup_{x\in X}} N_G(x))\setminus X$ and $N_G[X]=\cup_{x\in X} N_G[x]$. 
We also let $N_{G,2}[v]=N_G[N_G[v]]$, which is the set of all vertices with distance at most $2$ from $v$.
We say that a vertex $v$ not in $X$ is {\it adjacent to} $X$ if $v$ is adjacent to some vertex in $X$. 
For two disjoint subsets $X$ and $Y$ in $V(G)$, we denote by $E_G[X,Y]$ the set of edges in $G$ that join a vertex of $X$ and a vertex of $Y$, and we also let $\partial_G(X)=E_G[X,V(G)\setminus X]$.
When there is no confusion,
we often omit the subscript $G$ such as 
$d(v)$, $N(v)$, $N[v]$, $N_2[v]$, $N(X)$, $N[X]$, $E[X,Y]$, and $\partial(X)$.

\section{The weight function of a graph}

Let $G$ be a subcubic graph.
For each vertex $v$ in $G$, we define the {\it weight} of $v$  in $G$ as
\[\omega_G(v) =
		\begin{cases}
				9&\text{if } d(v)=0\\
				6-d(v)&\text{if } d(v)=i \text{ and }i>0.
		\end{cases}
\]
Given a set $X \subset V(G)$, we let \[\omega_G(X)=\sum_{v\in X}\omega_G(v).\]
Now we define the {\it weight} of $G$ as 
\[\scalebox{1.4}{$\omega$}(G)=9n_0(G)+5n_1(G)+4n_2(G)+3n_3(G)+2b(G).\]
A  \textit{$C_4$-component (resp. $K^*_4$-component)} is a connected component isomorphic to $C_4$ (resp. $K^*_4$). 
We denote by $b_1(G)$ (resp. $b_2(G)$) the  number of $C_4$-components (resp. $K^*_4$-components) of $G$. Note that $b(G)=b_1(G)+b_2(G)$ and it may happen 
$\scalebox{1.4}{$\omega$}(G)\neq \omega_G(V(G))$.

\begin{Fact}\label{Fact:WG:even}
Let $G$ be a subcubic graph with $n_0(G)=0$.
Then the following hold.
\begin{itemize} 
\item[\rm(i)]  $\scalebox{1.4}{$\omega$}(G)$ is an even integer.
\item[\rm(ii)] If $b(G)=0$, then $\scalebox{1.4}{$\omega$}(G)= 6|V(G)|-2|E(G)|$. 
\item[\rm(iii)] If $b(G)=0$  and  $n_0(G-v)=b(G-v)=0$ for a $3$-vertex $v$, then $\scalebox{1.4}{$\omega$}(G)=\scalebox{1.4}{$\omega$}(G-v)$.
\end{itemize}
\end{Fact}
\begin{proof}
Note that $n_1(G)+n_3(G)$ is the number of vertices with odd degree and so it  is even. Since $n_0(G)=0$,
$\scalebox{1.4}{$\omega$}(G)$ is an even integer, and so (i) holds. 
Since $n_0(G)=0$ and each edge decreases weight $1$ to each of its endpoints, (ii) holds.
It is easy to see that (iii) holds from (ii).    
\end{proof}

\begin{Lem}\label{Lem:b(G)=0}
Let $G$ be a subcubic graph such that $\scalebox{1.4}{$\omega$}(G) < 9\gamma_{b,2}(G)$ and
$\scalebox{1.4}{$\omega$}(H)\ge 9\gamma_{b,2}(H)$ for every proper subgraph $H$ of $G$. 
Then $G$ is connected, $\Delta(G)=3$, $|V(G)|\ge 8$, $b(G)=0$, and every $4$-cycle is an induced $4$-cycle.
\end{Lem}
\begin{proof}
If $G$ is not connected, then $G=G_1\cup G_2$ and so $\scalebox{1.4}{$\omega$}(G_1)+\scalebox{1.4}{$\omega$}(G_2)=\scalebox{1.4}{$\omega$}(G)$, and so by the assumption,
\[9\gamma_{b,2}(G)\le 9\gamma_{b,2}(G_1)+9\gamma_{b,2}(G_2)\leq \scalebox{1.4}{$\omega$}(G_1)+\scalebox{1.4}{$\omega$}(G_2)=\scalebox{1.4}{$\omega$}(G),\]
which contradicts the assumption that $\scalebox{1.4}{$\omega$}(G) < 9\gamma_{b,2}(G)$.
Thus $G$ is connected. If $\Delta(G)\le 2$, then $G$ is a cycle or a path and it is not difficult to check that $\scalebox{1.4}{$\omega$}(G) \ge 9\gamma_{b,2}(G)$, a contradiction. Thus $\Delta(G)=3$, and therefore $|V(G)|\ge 4$. 
Then it also follows that $\gamma_{b,2}(G)\ge 2$ and so $|V(G)|\ge 5$.
Suppose that $|V(G)|=5$. Then  $\gamma_{b,2}(G)=2$. Thus $3|V(G)|=15\le \scalebox{1.4}{$\omega$}(G)<18=9 \gamma_{b,2}(G)$.
Since $\scalebox{1.4}{$\omega$}(G)$ is even by Fact~\ref{Fact:WG:even}(i), $\scalebox{1.4}{$\omega$}(G)=16$ and so $G$ has four $3$-vertices and one $2$-vertex. Thus $G=K_4^*$ and so
$\scalebox{1.4}{$\omega$}(K^*_4)=4\cdot 3 + 4+2=18$, a contradiction.
Hence, $G$ is a connected graph with at least six vertices and so $b(G)=0$.

Let $C:abcda$ be a $4$-cycle with $ac\in E(G)$.
Since $G\neq K_4^*$, we have $n_0(G-a)=b(G-a)=0$, and therefore
$\scalebox{1.4}{$\omega$}(G-a)=\scalebox{1.4}{$\omega$}(G)$ by Fact~\ref{Fact:WG:even}(iii). 
Since a $2$-LD broadcast of $G-a$ is also a $2$-LD broadcast of $G$, we have a $2$-LD broadcast $f$ of $G$   such that $9f(G)\leq \scalebox{1.4}{$\omega$}(G)$, a contradiction. 
Thus every $4$-cycle of $G$ is an induced $4$-cycle. 

Now we will show that $|V(G)|\ge 8$.
Since $|V(G)|\geq 6$, $\scalebox{1.4}{$\omega$}(G)\ge 18$ and so $\gamma_{b,2}(G)>2$. Then the radius $r$ of $G$ is at least $3$.  
Suppose that $|V(G)|=6$. Since $r\ge 3$, 
$G$ has a spanning path $v_1v_2\cdots v_6$. 
Assigning $1$ to each of $v_2$ and $v_5$ and $0$ to the other vertices gives a $2$-LD broadcast with cost $2$, and so $\gamma_{b,2}(G)\le 2$, a contradiction.
Suppose that $|V(G)|=7$.   Take a $3$-vertex $u$.
If $G-N[u]$ is connected, then $\gamma_{b,2}(G)\le 2$ by assigning $1$ to each of $u$ and one vertex of $G-N[u]$ properly and assigning $0$ to the other vertices, a contradiction. 
If  $G-N[u]$ has at least three components, then the components in $G-N[u]$ are three trivial components and so $\gamma_{b,2}(G)\le 2$ by assigning $2$ to $u$ and $0$ to the other vertices, a contradiction.
Thus $G-N[u]$ has exactly two components $G_1=K_1$  and $G_2=K_2$. Let  $N(u)=\{v_1,v_2,v_3\}$, $V(G_1)=\{w\}$, and $V(G_2)=\{x_1,x_2\}$. (see Figure~\ref{fig:b(G)=0}).
\begin{figure}[h!]\centering
\includegraphics[page=2,width=3cm]{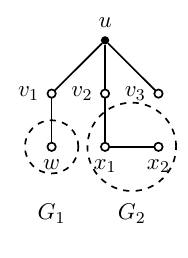} \vspace{-0.8cm}
\caption{The subgraph mentioned in Lemma~\ref{Lem:b(G)=0}} \label{fig:b(G)=0}
\end{figure}
If a vertex $v_i$ in $N(u)$ is adjacent both $G_1$ and $G_2$, then assigning $2$ to $v_i$ and $0$ to the other vertices defines a $2$-LD broadcast of $G$, a contradiction. Thus every vertex in $N(u)$ is adjacent at most one of $G_1$ and $G_2$. 
Therefore we may assume $wv_1\in E(G)$ and $x_1v_2\in E(G)$.
If $|E(G)|=6$ or $7$, then $\scalebox{1.4}{$\omega$}(G)\ge 28$ by Fact~\ref{Fact:WG:even} and $\gamma_{b,2}(G)\le 3$ by assigning $2$ to $u$ and $1$ to $x_2$, a contradiction.
Hence $|E(G)|\ge 8$ and so, by Fact~\ref{Fact:WG:even},  $\scalebox{1.4}{$\omega$}(G)\le 26$. 
Thus $G$ has at least two $3$-vertices.
By considering possibilities of edges joining $\{w,x_1,x_2\}$ and $\{v_1,v_2,v_3\}$, it is not difficult to check that we have $\gamma_{b,2}(G)\le 2$ in every case, which is a contradiction.
 \end{proof}

In the following proofs, as in the proof of Lemma~\ref{Lem:b(G)=0}, we often show that $\gamma_{b,2}(H)\le r$ for a subgraph $H$ of $G$ by assigning the values $0$, $1$, and $2$ to the vertices of $H$. 
Throughout, we typically mention only the vertices assigned the values $1$ or $2$, and we do not explicitly refer to the vertices assigned the value $0$.
 
\begin{Lem} \label{Lem:b_2(G)=0}
Let $G$ be a subcubic graph such that $\scalebox{1.4}{$\omega$}(G) < 9\gamma_{b,2}(G)$ and
$\scalebox{1.4}{$\omega$}(H)\ge 9\gamma_{b,2}(H)$ for every proper subgraph $H$ of $G$.
Then the following hold.
\begin{itemize}
    \item[\rm(i)] $b_2(G')=0$ for every subgraph $G'$ of $G$.
    \item[\rm(ii)] $b_1(G-e)=0$ for every edge $e$.
    \item[\rm(iii)] Every vertex has at most one  $1$-neighbor.

\end{itemize}
\end{Lem}
\begin{proof}
(i): By Lemma~\ref{Lem:b(G)=0}, every $4$-cycle is an induced cycle in $G$ and so $b_2(G')=0$.

To show (ii) and (iii),
we first prove the following.

\begin{Clm}\label{Clm:partial=1}
The graph $G$ contains no induced subgraph $G[X]$ such that 
$|X|\ge 3$, $|N(X)| \le 1$, and $G[X]$ has a vertex $x$ satisfying that $N_{G[X]}[x]=X$.
\end{Clm}
\begin{proof}
 By Lemma~\ref{Lem:b(G)=0}, $G$ is connected with $|V(G)|\geq 8$ and so  $N(X)\neq\emptyset$ for every subset $X\subset V(G)$ with $|X|\le 7$.
To the contrary, suppose that there is an induced subgraph $G[X]$ such that $|X|\ge 3$, $|N(X)|=1$, and $G[X]$ has a vertex $x$ satisfying that $N_{G[X]}[x]=X$.
Then $|X|\leq 4$.
We take such $G[X]$ with $|\partial(X)|$ as small as possible.  
For simplicity, let $H=G-X$.
Since $|N(X)|=1$, $H$ is connected. 
Since $H$ is a proper subgraph of $G$, $9\gamma_{b,2}(H)\leq \scalebox{1.4}{$\omega$}(H)$. 

Say $N(X)=\{y\}$. 
Recall that $\Delta(G)= 3$, $|V(G)|\ge 8$, $b(G)=0$ by Lemma~\ref{Lem:b(G)=0}. 
Then $y$ is a $2^-$-vertex in $H$ and $y$ is not an isolated vertex in $H$.
We will show that $b(H)=0$.
If $y$ is a $1$-vertex in $H$, then $H$ has a pendent edge and $b(H)=0$. 
Suppose that $y$ is a $2$-vertex in $H$ and $b(H)>0$. Then $H$ is $C_4$ or $K_4^*$. Since every $4$-cycle is an induced cycle by Lemma~\ref{Lem:b(G)=0}, $H=C_4$ and so 
$G$ has at least three $2$-vertices. Since $|V(G)|\ge 8$ by Lemma~\ref{Lem:b(G)=0},  $\scalebox{1.4}{$\omega$}(G)\geq 4\cdot 3+3\cdot 5=27$.
By assigning $1$ and $2$ to $x$ and $y$, respectively, we have $\gamma_{b,2}(G)\le 3$, and so  $\scalebox{1.4}{$\omega$}(G) \ge 27\ge 9\gamma_{b,2}(G)$, a contradiction.
Hence $b(H)=0$.

Recall that $y$ is not isolated in $H$. Then $|\partial(X)|\le 2$. In addition, since $x$ is a $2^+$-vertex and $ y$ has at most two neighbors in $N[x]$, it follows from $b(H)=b(G)=0$  that $\scalebox{1.4}{$\omega$}(G)-\scalebox{1.4}{$\omega$}(H)=\omega_G(X)-|\partial(X)|=\scalebox{1.4}{$\omega$}(G[X])-2|\partial(X)|$.  
We will show that $\scalebox{1.4}{$\omega$}(G)-\scalebox{1.4}{$\omega$}(H)\ge 9$.
If $|X|=4$, then  $\scalebox{1.4}{$\omega$}(G)-\scalebox{1.4}{$\omega$}(H)=\omega_G(X)-|\partial(X)|\ge 12-2=10$.
Suppose that $|X|=3$. Then $G[X]$ is a path or a triangle. 
If $G[X]$ is a triangle and $|\partial(X)|= 2$, 
then $G[X\cup\{y\}]$   contains a $4$-cycle that is not an induced cycle, a contradiction to Lemma~\ref{Lem:b(G)=0}.
Thus $G[X]$ is a path or $|\partial(X)| \le 1$, and so 
$\scalebox{1.4}{$\omega$}(G)-\scalebox{1.4}{$\omega$}(H)=\scalebox{1.4}{$\omega$}(G[X])-2|\partial(X)|\ge 9$.

Let $f'$ be a minimum $2$LD-broadcast in $H$. 
Then ${\rm cost}(f')=\gamma_{b,2}(H)$. 
We extend $f'$ to a function $f$ on $V(G)$ by assigning $1$ to $x$. Clearly, $f$ is a $2$LD-broadcast in $G$ and  $\gamma_{b,2}(G)\le {\rm{cost}}(f)=\gamma_{b,2}(H)+1$. 
Hence,
\[ \scalebox{1.4}{$\omega$}(G) < 9\gamma_{b,2}(G) \leq 9\gamma_{b,2}(H)+9 \leq \scalebox{1.4}{$\omega$}(H)+9  \leq \scalebox{1.4}{$\omega$}(G),\]
a contradiction.
\end{proof}

(ii):  To the contrary, suppose that
	there exists an edge $e$ of $G$ such that $G-e$ contains a $C_4$-component $C$.
   Since $|V(G)|\ge 5$, $e$ is a cut-edge in $G$. 
For the $2$-vertex $x$ of $C$ that is farthest from $e$, let $X=N[x]$. 
Then it follows that $|X|\ge 3$, $|N(X)|=1$, and $G[X]$ has vertex $x$ satisfying that $N_{G[X]}[x]=X$, a contradiction to Claim~\ref{Clm:partial=1}. 

(iii): To the contrary, suppose that $G$ contains a vertex $x$ with two  $1$-neighbors. Let $X=\{x\}\cup \{ v\in N(x) \mid d(v)=1\}$. 
Then   $|X|\ge 3$, $|N(X)|=1$, and $N_{G[X]}[x]=X$, a contradiction to Claim~\ref{Clm:partial=1}. \end{proof}

\section{Proof of Theorem~\ref{Thm:Main_subcubic}}

Suppose, to the contrary, that Theorem~\ref{Thm:Main_subcubic} is false.
We take a  minimal counterexample $G$ with respect to the number of vertices and edges.
By Lemma~\ref{Lem:b(G)=0}, $G$ is connected, $\Delta(G)=3$, $|V(G)|\ge 8$, $b(G)=0$, and  every $4$-cycle is an induced $4$-cycle.
In addition, $n_0(G)=0$ and so 
\begin{equation}\label{eq:first}
9\gamma_{b,2}(G)>\scalebox{1.4}{$\omega$}(G)=5n_1(G)+4n_2(G)+3n_3(G).\end{equation}

\begin{Lem}\label{Lem:existence_deg_1_2}
It holds that $9\gamma_{b,2}(G)-\scalebox{1.4}{$\omega$}(G)=1$ or $2$.
Therefore, the graph $G$ contains a $2^-$-vertex. 
\end{Lem}
\begin{proof} Since $G$ is not isomorphic to $K_{1,t}$, there is an edge $e$ in $G$ such that $e$ is not incident to any $1$-vertex. Consider $G':=G-e$.
By Lemma~\ref{Lem:b_2(G)=0}(i) and (ii),
$b(G')=0$, 
	and so $\scalebox{1.4}{$\omega$}(G')=\scalebox{1.4}{$\omega$}(G)+2$. 
	Since $|V(G)|=|V(G')|$ and $|E(G)|>|E(G')|$,
	$9\gamma_{b,2}(G')\leq \scalebox{1.4}{$\omega$}(G')$ by the choice of $G$.
Since $G'$ is a subgraph of $G$, 
	$\gamma_{b,2}(G) \leq \gamma_{b,2}(G')$.
	Then
	\[\scalebox{1.4}{$\omega$}(G) <9\gamma_{b,2}(G) \leq 9\gamma_{b,2}(G')  \leq \scalebox{1.4}{$\omega$}(G')=\scalebox{1.4}{$\omega$}(G)+2  \]
Hence $9\gamma_{b,2}(G)-\scalebox{1.4}{$\omega$}(G)=1$ or $2$. 

If $n_1(G)=n_2(G)=0$, then from \eqref{eq:first}, $\scalebox{1.4}{$\omega$}(G)=3n_3(G)=3|V(G)|$, which contradicts the fact that $9\gamma_{b,2}(G)=\scalebox{1.4}{$\omega$}(G)+k$ for some $k\in \{1,2\}$.
Therefore $n_1(G)\neq 0$ or $n_2(G)\neq 0$, and so $G$ has a $2^-$-vertex.
\end{proof}
 
Let $f_i$ be a function defined on $V_i$ for each $i\in\{1,2\}$.
When $V_1\cap V_2=\emptyset$,
we denote by $f_1 \cup f_2$ the function $f$ defined by $f(x)=f_i(x)$ if $x\in V_i$ for some $i\in\{1,2\}$. 

\begin{Lem}\label{Lem:cutedge4}
Let $S$ be an edge cut of $G$ such that $G-S=G_1\cup G_2$.
For each $i\in \{1,2\}$, let $r_i\in \{0,1,\ldots,8\}$ such that
$\scalebox{1.4}{$\omega$}(G_i)\equiv r_i \pmod{9}$.
If $n_0(G-S)=b_1(G-S)=0$, then $r_1+r_2\le 2|S|-1$.
\end{Lem}

\begin{proof} 
For each $i\in \{1,2\}$, let $q_i$ be a nonnegative integer such that $\scalebox{1.4}{$\omega$}(G_i)=9q_i+r_i$. 
Suppose that $n_0(G-S)=b_1(G-S)=0$.
Then by Lemma~\ref{Lem:b_2(G)=0}(i) $b(G_1)=b(G_2)=0$.
By the minimality of $G$, for $i=1,2$, there is a $2$-LD broadcast $f_i$ of $G_i$ with cost at most $q_i$. 
Then $f_1 \cup f_2$ is a $2$-LD broadcast of $G$.
Therefore $9(q_1+q_2)\ge 9\gamma_{b,2}(G)> \scalebox{1.4}{$\omega$}(G)=\scalebox{1.4}{$\omega$}(G_1)+\scalebox{1.4}{$\omega$}(G_2)-2|S|=9(q_1+q_2)+r_1+r_2-2|S|$ and so $r_1+r_2-2|S|<0$. 
\end{proof}

Let $X$ be a subset  of $V(G)$.
Let $\mathcal{C}(X)$ be the union of $X$ and the vertex sets of isolated vertices or $C_4$-components in $G-X$. 
For a subgraph $X$ of $G$, a vertex $v$ is said to be {\it adjacent} to $X$ if $v$ is adjacent to some vertex of $V(X)$, and $\mathcal{C}(X)$ means $\mathcal{C}(V(X))$.
Note that 
$G-\mathcal{C}(X)$ has neither isolated vertex nor  $C_4$-component, and therefore $\mathcal{C}(\mathcal{C}(X))=\mathcal{C}(X)$.  

\begin{Fact} \label{Fact:main_inequality} 
For $X \subsetneq V(G)$, \[ |\partial(X)|+3n_0(G-X)> \omega_G(X)-2b_1(G-X)-9\gamma_{b,2}(G[X])+q \ge \omega_G(X)-2b_1(G-X)-9\gamma_{b,2}(G[X]),\] where $q=\scalebox{1.4}{$\omega$}(G-X)-9\gamma_{b,2}(G-X)$. 
Moreover,   
$|\partial(\mathcal{C}(X))|-q > \omega_G(\mathcal{C}(X))-9\gamma_{b,2}(G[\mathcal{C}(X)])$, where  $q=\scalebox{1.4}{$\omega$}(G-\mathcal{C}(X))-9\gamma_{b,2}(G-\mathcal{C}(X))$.
\end{Fact}
\begin{proof} 
Since $X \subsetneq V(G)$,
$V(G-X)\neq \emptyset$. By the minimality of $G$, note that $q\ge 0$.
By definition,
\[ |\partial(X)|+3n_0(G-X) =\sum_{v\in V(G)-X} (\omega_{G-X}(v)-\omega_G(v))  = (\scalebox{1.4}{$\omega$}(G-X)-2b(G-X) ) - \omega_G(V(G)-X).
\]
By Lemma~\ref{Lem:b(G)=0}, $b(G)=0$ and $b_2(G-X)=0$. Together with $\scalebox{1.4}{$\omega$}(G)=\omega_G(X)+\omega_G(V(G)-X)$, it follows that
\begin{eqnarray}\label{eq:1}
    &&|\partial(X)|+3n_0(G-X)= \scalebox{1.4}{$\omega$}(G-X) -2b_1(G-X)+\omega_G(X)-\scalebox{1.4}{$\omega$}(G).
\end{eqnarray}
On the other hand, since $\gamma_{b,2}(G)\leq \gamma_{b,2}(G-X)+\gamma_{b,2}(G[X])$,
\begin{eqnarray*} \scalebox{1.4}{$\omega$}(G) < 9\gamma_{b,2}(G)  \leq 9\gamma_{b,2}(G-X)+9\gamma_{b,2}(G[X]) \le \scalebox{1.4}{$\omega$}(G-X)-q+9\gamma_{b,2}(G[X]),\end{eqnarray*}
and so,
$ \scalebox{1.4}{$\omega$}(G-X)-\scalebox{1.4}{$\omega$}(G)-q> -9\gamma_{b,2}(G[X])$.
Then the fact follows from \eqref{eq:1}. 
Since $G-\mathcal{C}(X)$ has neither isolated vertex nor  $C_4$-component, 
$n_0(G-\mathcal{C}(X))=b_1(G-\mathcal{C}(X))=0$.
If $\mathcal{C}(X)\subsetneq V(G)$, then
the `moreover' part holds. 
Suppose $\mathcal{C}(X)=V(G)$.
Since $b(G)=0$, $\omega_G(\mathcal{C}(X))=\scalebox{1.4}{$\omega$}(G)$.
In addition, by Lemma~\ref{Lem:existence_deg_1_2},
$\scalebox{1.4}{$\omega$}(G)-9\gamma_{b,2}(G)=-1$ or $-2$.
and so the `moreover' part also holds.
\end{proof}

\begin{Lem}\label{Lem:removal-two-edges}
For every $4$-cycle $C$, $|\partial(C)|\ge 3$, and therefore $b_1(G-v)=0$ for every vertex $v$.
\end{Lem} 

\begin{proof}
It is sufficient to show that $b_1(G-\{e_1,e_2\})=0$ for every two edges $e_1,e_2$. To the contrary, suppose that there are two edges $u_1v_1$ and $u_2v_2$ of $G$ such that $G-\{u_1v_1,u_2v_2\}$ contains a $C_4$-component $C$. We may assume $\{v_1,v_2\}\subset V(C)$.  Clearly $v_1\neq v_2$.
By Lemma~\ref{Lem:b_2(G)=0}(ii), 
$\partial(C)=\{u_1v_1,u_2v_2\}  $.

\noindent(Case 1) Suppose that $u_1 = u_2$. 
Then $b_1(G-u_1)>0$ and $N(V(C))=\{u_1\}$. 
Since $|V(G)|\ge 8$, there is a neighbor $u$ of $u_1$ not on $C$, and $u$ is a $2^+$-vertex.  
Let $X=V(C)\cup\{u,u_1\}$. 
Then $\scalebox{1.4}{$\omega$}(G[X])=22\equiv 4\pmod{9}$ and $1\le |\partial{(X)}|\le 2$. By Lemma~\ref{Lem:cutedge4}, $G-X$ has an isolated vertex or a $C_4$-component. 
Suppose that $G-X$ has no isolated vertex. Then $G-X$ has a $C_4$-component $C'$.
By Lemma~\ref{Lem:b_2(G)=0}(ii),   $G$ has exactly 10 vertices, and   $30<\scalebox{1.4}{$\omega$}(G)$.  In addition, $\gamma_{b,2}(G)\le3$ by assigning $2$ to $u_1$ and $1$ to some vertex of $C'$, a contradiction. 
Thus $G-X$ has an isolated vertex. 
Therefore $G-X$ has a unique isolated vertex $w$  by Lemma~\ref{Lem:b_2(G)=0}(iii).
For $Y=X\cup\{w\}$,  $|\partial{(Y)}|=1$, $\scalebox{1.4}{$\omega$}(G[Y])=26\equiv 8\pmod{9}$ and $G-Y$ has no isolated vertex. In addition, by Lemma~\ref{Lem:b_2(G)=0}(ii), $G-Y$ has no $C_4$-component, a contradiction to
Lemma~\ref{Lem:cutedge4}.

\noindent(Case 2) Suppose that $u_1\neq u_2$.
Now we consider the new graph $G'$ obtained from $G$ by deleting the vertices of $C$ and adding a new vertex $w$ and two edges $wu_1$ and $wu_2$, that is, 
    \[V(G')=V(G-C) \cup \{w\}\quad \text{and} \quad E(G')=E(G-C)\cup \{wu_1,wu_2\}.\]
    Then $G'$ is subcubic.   
Since $|V(G)|\ge 8$, $|V(G')|\ge 5$. 
In addition, $G'$ is connected and so $b_1(G')=0$.
Furthermore, since every $4$-cycle of $G$ is an induced cycle by Lemma~\ref{Lem:b(G)=0},  $b_2(G')=0$.  
We can check that  $\omega_G(C)=14$ and so  
$\scalebox{1.4}{$\omega$}(G')=\scalebox{1.4}{$\omega$}(G)-10$. 
By the minimality of $G$,  
\[9\gamma_{b,2}(G')\leq \scalebox{1.4}{$\omega$}(G')= \scalebox{1.4}{$\omega$}(G)-10 < 9\gamma_{b,2}(G)-10 < 9(\gamma_{b,2}(G)-1) \]
and so $\gamma_{b,2}(G')+1 <\gamma_{b,2}(G)$.
We will reach a contradiction by constructing a $2$-LD broadcast $f^*$ in $G$ with cost  $\gamma_{b,2}(G')+1$. 
There is a minimum $2$-LD broadcast $f$ of $G'$ such that $f(w)\neq 2$ by the following procedure: If $f(w)=2$, then we redefine the values on $N[w]$ so that $f(w):=0$ and $f(x):=\max\{f(x),1\}$ if $x\in N(w)$.
Then either $f(w)=1$ or $f(w)=0$. 
We extend $f|_{V(G')\setminus\{w\}}$ as a function on $V(G)$ by assigning 0 to each vertex of $V(C)$, and use the same notation $f$ for this extended function.

Let $Y$ be the set of vertices in $G$ that do not hear from some vertex under $f$ in $G$. Clearly, $Y\subset V(C)\cup\{u_1,u_2\}$.
We define a function $f^*:V(G)\to \mathbb{Z}_{\ge 0}$ so that
\[f^*(\nu)= \begin{cases}
				f(w)+1 &\text{if $\nu=v^*$}\\
				f(\nu) 
                &\text{otherwise,}
		\end{cases} 
		\] 
where a vertex $v^*$ is defined as follows: 
If $f(w)=1$, then let $v^*$ be a vertex on $C$ 
such that the distance from it to each of
 $u_1$ and $u_2$ is at most two in $G$ (see the first and the third figures of Figure~\ref{fig:v*}). 
Suppose that $f(w)=0$. Then we may assume that $u_1, v_1\not\in Y$. 
If $u_2\not\in Y$, then let $v^*$ be the vertex nonadjacent to $v_1$ on $C$  (see the second and the fourth figures of Figure~\ref{fig:v*}). 
If $u_2\in Y$, then we have $f(u_1)=2$ and then let $v^*=v_2$ (see the second and the third figures of Figure~\ref{fig:v*}). Then we can check that $f^*$ is a $2$-LD broadcast  in $G$ with cost $\gamma_{b,2}(G')+1$, a contradiction.        
\begin{figure}[h!]
    \centering
    \includegraphics[page=3, width=0.9\textwidth]{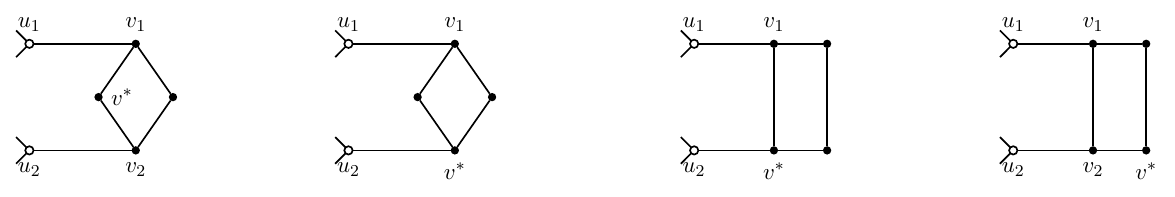}
    \caption{The choices of the vertex $v^*$}
    \label{fig:v*}
\end{figure}
\end{proof}

\subsection{Basic observations on the degree of a vertex}

This subsection aims to show the following lemma:

\begin{Lem} \label{Lem:key3.1} The following hold. 
\begin{itemize}
\item[\rm(i)]      $|V(G)|\geq 10$, and therefore $\scalebox{1.4}{$\omega$}(G)\ge 34$ and $\gamma_{b,2}(G)\ge 4$.
\item[\rm(ii)]   Every vertex on a triangle is a $(3,3,3)$-vertex.
\item[\rm(iii)] It holds that $\delta(G)=2$, and every vertex is a $(3,3)$-vertex or a $(2,3,3)$-vertex or a $(3,3,3)$-vertex.
\end{itemize}
\end{Lem}

In what follows, we collect the necessary observations to prove the above lemma. 
 
\begin{Lem}\label{Lem:no-(2,2)} 
There is no path $P:uvw$ in $G$ 
whose vertex degrees in $G$ sum to at most six.
Therefore, no $(2^-,2^-)$-, $(1,3^-)$-, or $(1,2^-,3^-)$-vertex exists.
\end{Lem}
\begin{proof}
To the contrary, suppose that there is a path $P:uvw$ in $G$ such that 
 $d(u)+d(v)+d(w)\le 6$.
By the degree condition, $|\partial(P)|\le 2$.
By Lemma~\ref{Lem:removal-two-edges} and from the fact that $|V(G)|\ge 8$, it follows that  $b_1(G-V(P))=0$ 
and $P$ is an induced path in $G$. 
Thus $\scalebox{1.4}{$\omega$}(P)=14\equiv 5\pmod{9}$ and so, by Lemma~\ref{Lem:cutedge4}, $G-V(P)$ has an isolated vertex $x$.
Since $|V(G)|\ge 8$,  
$x$ is a unique isolated vertex in $G-V(P)$, $d(x)=1$, and there is a unique vertex $y$ that is adjacent to $\mathcal{C}(P)$ with $y\not\in \mathcal{C}(P)$. 
For simplicity, let 
$R=\partial(\mathcal{C}(P)\cup \{y\}).$ 
Then $|R|\le 2$. 
Let $G_1$ be the connected component of $G-R$ containing $P$.
Since $|V(G)|\ge 8$, if $G-R$ has an isolated vertex, then it has unique isolated vertex $z$. In this case,  we redefine $G_1$ by adding $z$ to $G_1$, and then redefine $R=\partial(G_1)$.  It also holds that $|R|\le 2$.
Then $G_1$ has $5$
 or $6$ vertices, and so $\scalebox{1.4}{$\omega$}(G_1)=22$ or $26$ by Fact~\ref{Fact:WG:even}(ii).
Then $\scalebox{1.4}{$\omega$}(G_1)\equiv 4$ or $8\pmod{9}$, and so $b_1(G-R)>0$
by Lemma~\ref{Lem:cutedge4}. Thus $G-y$ contains a $C_4$-component, which contradicts Lemma~\ref{Lem:removal-two-edges}
\end{proof}

\begin{Lem}\label{Lem:partial=2}
  If a $3$-vertex $v$ has a $2^-$-neighbor, then $G-N[v]$ has an isolated vertex  or $|\partial(N[v])|\ge 5$.
\end{Lem}

\begin{proof} Suppose that a $3$-vertex $v$ has a $2^-$-neighbor, $G-N[v]$ has no isolated vertex, and $|\partial(N[v])|\le 4$.
We let $X=N[v]$.
Note that $\omega_G(X)\ge 13$ since $X$ has a $2^-$-vertex and $|X|= 4$.
Then it follows from Fact~\ref{Fact:main_inequality}, $4\ge |\partial(X)|>13-2b_1(G-X)-9$ and so $b_1(G-X)>0$, that is, $G-X$ has a $C_4$-component $C$.
By Lemma~\ref{Lem:removal-two-edges}, $|\partial(C)|\geq 3$ , at least two vertices of $C$ belong to $N(X)$, and $C$ is the only $C_4$-component in $G-X$.
Hence $\mathcal{C}(X)=X\cup V(C)$.
In addition,  $G[\mathcal{C}(X)]$ has a spanning subgraph isomorphic to one of the graphs in Figure~\ref{fig:Lem:partial=2}.
Then $\gamma_{b,2}(\mathcal{C}(X))\le 2$, 
$|\partial(\mathcal{C}(X))|\leq 1$ and $\omega_G(\mathcal{C}(X))\ge 24$, which contradicts Fact~\ref{Fact:main_inequality}.
\begin{figure}[h!]\centering
\includegraphics[page=4, width=0.45\textwidth]{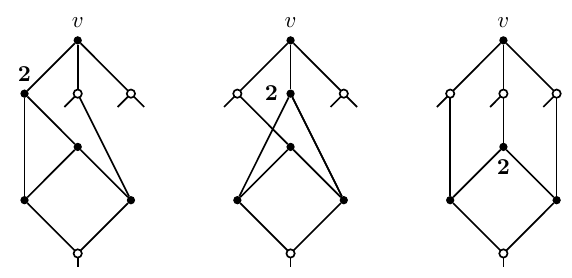}
\caption{Subgraphs mentioned in Lemma~\ref{Lem:partial=2}} \label{fig:Lem:partial=2}
\end{figure}
\end{proof}

\begin{Lem}\label{Lem:no_pendant} It holds that $\delta(G)=2$.
\end{Lem}
\begin{proof}
We first prove the following claim.

\begin{Clm}\label{clm:partial_rad_2}
Let $v$ be a vertex in $G$ with a $1$-neighbor $u$ and
$X$ be a set of vertices containing $u$ and $v$ such that every vertex in $X$ has distance at most two from $v$ in $G[X]$. If $\omega_G(X)\ge 21$,  then $|\partial(X)|\ge 4$ or $G-X$ has an isolated vertex.
\end{Clm} 
\begin{proof}
Suppose, to the contrary, that $\omega_G(X)\ge 21$, $|\partial(X)|\le 3$ and $G-X$ has no isolated vertex. 
If $b_1(G-X)=0$, then Fact~\ref{Fact:main_inequality} implies that 
$3\ge |\partial(X)|>21-18$, a contradiction. 
Thus $b_1(G-X)>0$ and so $G-X$ has a $C_4$-component $C$.
By Lemma~\ref{Lem:removal-two-edges}, $C$ is incident to all three edges of $\partial(X)$, $|\partial(X)|=3$ and so $C$ has one $2$-vertex.
Then $V(G)=X\cup V(C)$ and $\gamma_{b,2}(G)\le 4$.
Since $\scalebox{1.4}{$\omega$}(G)\ge 21+13=34$, $\gamma_{b,2}(G)=4$.
Note that $\scalebox{1.4}{$\omega$}(G)$ is even by Fact~\ref{Fact:WG:even}(i).
By Lemma~\ref{Lem:b(G)=0},  $\scalebox{1.4}{$\omega$}(G)=\omega_G(V(G))=34$, and so $|V(G)|\le 10$.
If $V(C)\cap N_2[v]\neq \emptyset$,
then $\gamma_{b,2}(G)\leq 3$ by assigning $2$ to $v$ and $1$ to some vertex in $C$, a contradiction.
Thus $V(C)\cap N_2[v]=\emptyset$, and therefore, $X=N_2[v]$.
By Lemma~\ref{Lem:removal-two-edges}, $|X\setminus N[v]|\ge 2$, and therefore 
$|V(G)|=10$
and $|X\setminus N[v]|=2$. Take a vertex $v^*$ on $C$ so that $v^*$ has distance at most two from $v^*$ to each of two vertices in $X\setminus N[v]$.
Therefore $\gamma_{b,2}(G)\leq 3$ by assigning $2$  and $1$ to $v^*$ and $v$, respectively, a contradiction.
\end{proof}

Now suppose to the contrary that $\delta(G)=1$.
By Lemma~\ref{Lem:no-(2,2)}, there is a $(1,3,3)$-vertex in $G$.

\begin{Clm}\label{Clm:133} 
There is no edge between two $(1,3,3)$-vertices.  
\end{Clm}
\begin{proof}
Let $V_3$ be the set of $(1,3,3)$-vertices. 
Take a longest path $v_1\ldots v_t$ of $G[V_3]$. Let $u_i$ be the $1$-neighbor of $v_i$. Let $X=\{v_1,\ldots,v_t\}\cup \{u_1,\ldots,u_t\}$. 
Suppose to the contrary that $t\ge 2$. 
Let $w_1$ and $w_t$ be the neighbors of $v_1$ and $v_t$ not in $X$, respectively.
Then $|\partial(X)|=2$ and so $b_1(G-X)=0$ by Lemma~\ref{Lem:removal-two-edges}.
We can check $\omega_G(X)=8t$.
Let $t=3q+r$ for  some integers $q$ and $r$ with $r\in\{0,1,2\}$. 
Then $\gamma_{b,2}(G[X])\le 2q + r$ since we define a $2$-LD broadcast $f$ of $G[X]$ by assigning $2$ to $v_{3i+2}$  for each nonnegative integer $i$ such that $3i+2\leq t$ and by assigning $1$ to $v_t$ only when $r=1$. 
Since $v_1$ and $v_t$ are $(1,3,3)$-vertices, $G-X$ has no isolated vertex.
Then by Fact~\ref{Fact:main_inequality}, $2\ge |\partial(X)|>8t-9(2q+r)=6q-r$ and therefore $q=0$. 
Thus $t=2$. 

Let $Y=N[v_1]\cup N[v_2]$. Note that $Y$ is a subset of $N_2[v_i]$ for each $i\in \{1,2\}$.
If $w_1=w_2$, 
 then by Lemma~\ref{Lem:removal-two-edges}, $b_1(G-Y)=0$ and we can check that $\partial(Y)=1$, $\scalebox{1.4}{$\omega$}(G[Y])=20\equiv 2\pmod{9}$, a contradiction to Lemma~\ref{Lem:cutedge4}. 
Thus $w_1 \neq w_2$. 
Suppose that $G-Y$ has an isolated vertex $w$. 
If $d(w)=1$, then $w_1$ or $w_2$ is a $(1,3,3)$-vertex by Lemma~\ref{Lem:no-(2,2)}, which contradicts the maximality of $t$.
Thus $d(w)\ge 2$.
Thus $N(w)=\{w_1,w_2\}$.
 Thus, if $G-Y$ has an isolated vertex other than $w$, then $G$ is determined and so we can check that $\gamma_{b,2}(G)=2$ and $\scalebox{1.4}{$\omega$}(G)=30$, a contradiction.
Hence $w$ is the only isolated vertex in $G-Y$.
Let $Z=Y\cup\{w\}$. Then
  $n_0(G-Z)=0$, $|\partial(Z)|\le 2$, and $\omega_G(Z)=26\ge 21$. This contradicts Claim~\ref{clm:partial_rad_2}. Hence, $G-Y$ has no isolated vertex. 

Since $\scalebox{1.4}{$\omega$}(G[Y])=26\equiv 8\pmod{9}$, by Lemma~\ref{Lem:cutedge4}, $b_1(G-Y)>0$.
By Lemma~\ref{Lem:removal-two-edges}, $b_1(G-Y)=1$. Let $C$ be the $C_4$-component in $G-Y$. Note that $\mathcal{C}(Y)=Y\cup V(C)$ and $\omega_G(\mathcal{C}(Y))\ge 34$,
By Lemma~\ref{Lem:removal-two-edges}, $|\partial(\mathcal{C}(Y))|\leq1$.
Since $V(C)\cap N_2[v_i]\neq\emptyset$ for some $i\in \{1,2\}$, $\gamma_{b,2}(G[\mathcal{C}(Y)])\leq 3$ by assigning $2$ to $v_i$ and $1$ to some vertex on $C$, a contradiction to
Fact~\ref{Fact:main_inequality}.
\end{proof}

Let $v$ be a $(1,3,3)$-vertex, and 
let $N(v)=\{v_1,v_2,v_3\}$ where $d(v_1)=1$.
By letting $X=N[v]=\{v,v_1,v_2,v_3\}$, we have $|\partial(X)|\le 4$ and so 
by Lemma~\ref{Lem:partial=2}, $G-X$ has an isolated vertex.
By Lemma~\ref{Lem:no-(2,2)} and Claim~\ref{Clm:133},
each isolated vertex in $G-X$ has degree $2$.
Thus, since $|V(G)|\geq 8$ by Lemma~\ref{Lem:b(G)=0}, $G-X$ has a unique isolated vertex $w$ with $N(w)=\{v_2,v_3\}$. Then $v_2v_3\not\in E(G)$ by Lemma~\ref{Lem:b(G)=0}.
We follow the vertex labeling shown in Figure~\ref{fig:Lem:no_pendant}.
By Lemma~\ref{Lem:no-(2,2)},
$w_2$ and $w_3$ are $2^+$-vertices.

Let $X_1=X\cup\{w,w_2\}$.  
If $w_2=w_3$, then we can check $\scalebox{1.4}{$\omega$}(G[X_1])=22$ and  $|\partial(X_1)|\le 1$, which is a contradiction to Lemmas~\ref{Lem:b(G)=0}, \ref{Lem:cutedge4}, and \ref{Lem:removal-two-edges}.
Therefore, $w_2 \neq w_3$. 
Then $\scalebox{1.4}{$\omega$}(X_1) \ge 21$ and 
$|\partial(X_1)|\le 3$. Thus $G-X_1$ has an isolated vertex by  Claim~\ref{clm:partial_rad_2}. 
 By Lemma~\ref{Lem:b_2(G)=0}(iii), $G-X_1$ has at most two isolated vertices.
Since $G$ has no $(1,2,3)$-vertex by Lemma~\ref{Lem:no-(2,2)},
$G-X_1$ has only one isolated vertex $w'$ and $N(w')=\{w_2\}$ or $\{v_3,w_2\}$. 
Thus $w'$ is adjacent to $w_2$. 

Let $X_2=X_1\cup \{w'\}$. Then $\omega_G(X_2)\ge 21$ and $|\partial(X_2)|\le 2$. By Lemma~\ref{Lem:removal-two-edges}, $b_1(G-X_2)=0$. 
Since $w'$ is the unique isolated vertex in $G-X_1$,
$n_0(G-X_2)=0$.
In addition, $\gamma_{b,2}(G[X_2])\le 2$ by assigning $2$ to $v_2$, a contradiction to Fact~\ref{Fact:main_inequality}. 
 Hence we have shown that $\delta(G)\ge 2$.  By Lemma~\ref{Lem:existence_deg_1_2}, $\delta(G)=2$.
\end{proof}

\begin{figure}[h!]\centering
\includegraphics[page=5]{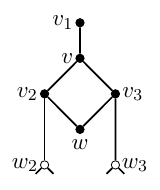}
\caption{Subgraphs mentioned in Lemma~\ref{Lem:no_pendant}}
\label{fig:Lem:no_pendant}
\end{figure}

The following Lemma~\ref{0Lem:no-triagnle-vertex-of-degree2-new} implies Lemma~\ref{Lem:key3.1}(ii)  immediately.
 
\begin{Lem}\label{0Lem:no-triagnle-vertex-of-degree2-new}
Suppose that a $3$-vertex $v$ has a $2^-$-neighbor.
Then 
$|\partial(N[v])|\ge 5$.
Therefore, there is no $(2^-,2^-,3^-)$-vertex,  and 
every vertex on a triangle is a $(3,3,3)$-vertex.
\end{Lem}
\begin{proof}
Let $X=N[v]$.
Suppose to the contrary that $|\partial(X)|\le 4$. 
By Lemma~\ref{Lem:partial=2}, $G-X$ has an isolated vertex $x$. 
Then $d(x)\ge 2$ by Lemma~\ref{Lem:no_pendant} and so, by  Lemma~\ref{Lem:b(G)=0}, $x$ is the only isolated vertex in $G-X$.
 If $d(x)=3$, then $N(v)$ is an independent set by Lemma~\ref{Lem:b(G)=0} and by the fact that $|\partial(X)|\le 4$, at most one vertex in $N(v)$ is $3$-vertex, which contradicts Lemma~\ref{Lem:removal-two-edges}.
Thus $d(x)=2$.
Let $N(x)=\{v_1,v_2\}$.
Considering the $4$-cycle containing $x$ and $v$, from
Lemma~\ref{Lem:removal-two-edges},  both $v_1$ and $v_2$ are $3$-vertices. 
Let $v_3$ be the $2$-neighbor of $v$.
Since $|\partial(X)|\le 4$, we may assume that $v_2v_3\in E(G)$.
Then for the set $Y=N_2[v]$, it follows that $|Y|=6$, $\omega_G(Y)\ge 20$, $|\partial(Y)|\le 2$, and $\gamma_{b,2}(G[Y])\le 2$.
By Lemma~\ref{Lem:removal-two-edges}, $b_1(G-Y)=0$.
By Lemma~\ref{Lem:no_pendant}, $n_0(G-Y)=0$.
It is a contradiction to Fact~\ref{Fact:main_inequality}.
\end{proof}

We now prove Lemma~\ref{Lem:key3.1}(i).

\begin{proof}[Proof of Lemma~\ref{Lem:key3.1}(i)]
  To the contrary, suppose that $|V(G)|<10$.
Then by Lemma~\ref{Lem:b(G)=0},
$|V(G)|=8$ or $9$.
By Lemma~\ref{Lem:no_pendant},
$G$ contains a $2$-vertex.  
In addition, since the number of odd vertices is even, the number of $3$-vertices is even.
If $|V(G)|=9$, then $G$ has at least four $3$-vertices by Lemma~\ref{Lem:no-(2,2)},  and so $28\le \scalebox{1.4}{$\omega$}(G) \le 32$, which  contradicts Lemma~\ref{Lem:existence_deg_1_2}. 
Then $|V(G)|=8$ and so $G$ contains at least two $2$-vertices.
If $G$ contains at most four $3$-vertices, then $28\le \scalebox{1.4}{$\omega$}(G)\le 32$, which contradicts Lemma~\ref{Lem:existence_deg_1_2}.
Thus $G$ contains exactly six $3$-vertices and two $2$-vertices.
Then $\scalebox{1.4}{$\omega$}(G)=26$ and so, by Lemma~\ref{Lem:existence_deg_1_2}, $\gamma_{b,2}(G)=3$.
We contract the $2$-vertices of $G$ to obtain a cubic graph $H$ with exactly $6$ vertices. Suppose that $H$ is not simple. Then $H$ has a parallel edge.
Thus by Lemma~\ref{0Lem:no-triagnle-vertex-of-degree2-new}, the two $2$-vertices are on some $4$-cycle of $G$, which contradicts Lemma~\ref{Lem:removal-two-edges}.
Hence $H$ is simple.
It is well-known that $H$ is isomorphic to $K_{3,3}$ or the prism graph.
Then Figure~\ref{fig:6vertices:all} shows all cases, and it is easy to find a $2$-LD broadcast with cost two in each case, a contradiction.
Thus $|V(G)|\ge 10$.

\begin{figure}[h!]
\centering
\includegraphics[width=12cm,page=6]{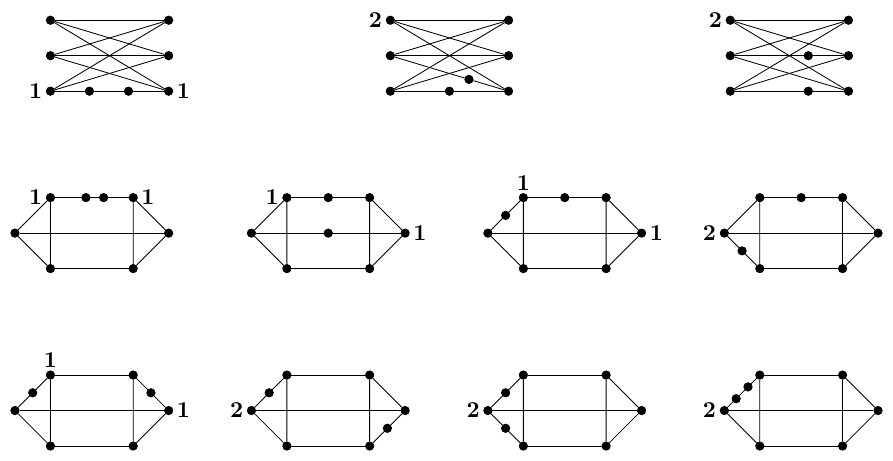}
\caption{All possibilities of $G$, where the numbers indicate the positive values of a $2$-LD broadcast }\label{fig:6vertices:all}
\end{figure}
  
By Fact~\ref{Fact:WG:even}(i), $\scalebox{1.4}{$\omega$}(G)$ is even. Thus $|V(G)|\ge 10$ implies $\scalebox{1.4}{$\omega$}(G)\ge 32$ and so $\gamma_{b,2}(G)\ge 4$.
Then by Lemma~\ref{Lem:existence_deg_1_2}, $\scalebox{1.4}{$\omega$}(G)\ge 34$.
\end{proof}

We finish this subsection with the proof of Lemma~\ref{Lem:key3.1}(iii).

\begin{proof}[Proof of Lemma~\ref{Lem:key3.1}(iii)]
By Lemma~\ref{Lem:no_pendant},
 $\delta(G)=2$.
Take a $3$-vertex $v$ in $G$.
If $v$ has a $2^-$-neighbor,
then $v$ is a $(2,3,3)$-vertex by Lemma~\ref{0Lem:no-triagnle-vertex-of-degree2-new}.
If $v$ has no $2^-$-neighbor,
then $v$ is a $(3,3,3$)-vertex. 

It remains to show that every $2$-vertex in $G$ is a $(3,3)$-vertex. 
Suppose to the contrary that $G$ contains two adjacent $2$-vertices $v_1$ and $v_2$. 
For each $i\in \{1,2\}$, let $u_i$ be the neighbor of $v_i$ other than $v_{3-i}$.  By Lemma~\ref{Lem:no-(2,2)}, each $u_i$ is a $3$-vertex.
By Lemma~\ref{0Lem:no-triagnle-vertex-of-degree2-new}, $u_1\neq u_2$.
By Lemma~\ref{Lem:removal-two-edges}, $u_1u_2 \not \in E(G)$.
If $u_1$ and $u_2$ have two common neighbors $w$ and $w'$, then for $X=\{v_1,v_2,u_1,u_2,w,w'\}$, it holds that $1\le|\partial(X)|\le 2$, $\scalebox{1.4}{$\omega$}(G[X])=22$, $b(G-X)=0$, and $n_0(G-X)=0$ by Lemmas~\ref{Lem:removal-two-edges}, \ref{Lem:key3.1}(i), and   \ref{Lem:no_pendant}, which contradicts Lemma~\ref{Lem:cutedge4}.
Thus $u_1$ and $u_2$ have at most one common neighbor.
For each $i\in \{1,2\}$, let $w_i$ be a neighbor of $u_i$ other than $v_i$ that is not a neighbor of $u_{3-i}$.
By Lemma~\ref{0Lem:no-triagnle-vertex-of-degree2-new},  each $w_i$ is a $3$-vertex. 
We take $w_1$ and $w_2$ so that $w_1w_2$ is not an edge in $G$, if possible. Let $w'_i$ be the neighbor of $u_i$ other than $v_i$ and $w_i$.

Let $G'$ be the graph obtained from $G-\{v_1,v_2,u_1\}$ by adding an edge $w_1u_2$, that is, $G'=G-\{v_1,v_2,u_1\}+w_1u_2$.
Then  $d_{G'}(w_1)=d_{G'}(u_2)=3$.
It is clear that $G'$ has no isolated vertex, since $d_{G'}(w'_i)=d_G(w'_i)=3$ by Lemma~\ref{0Lem:no-triagnle-vertex-of-degree2-new}.

In the following, we will show that $b(G')=0$. 
Suppose to the contrary that $b(G')>0$. Note that $G'$ has no $C_4$-component  by Lemma~\ref{Lem:removal-two-edges}.
Then $G'$ contains a component $K$ isomorphic to $K_4^*$, and moreover $K$ contains the edge $w_1u_2$, since every $4$-cycle of $G$ is an induced $4$-cycle.  Since $|V(G)|\ge 10$, $w'_1$ is not a vertex of $K$. Then by the choice of $w_1$ and $w_2$, it follows that $w_1w_2\not\in E(G)$. Note that in the graph $K$, $w_1,w_2,u_2,w'_2$ form a $K_4$ with a missing edge. Then $u_2,w'_2,w_2$ form a triangle in $G$ and $u_2$ has a $2$-neighbor in $G$, a contradiction to Lemma~\ref{0Lem:no-triagnle-vertex-of-degree2-new}.
Hence $b(G')=0$. 

Then $\scalebox{1.4}{$\omega$}(G')= \scalebox{1.4}{$\omega$}(G)-10$.
In addition, by the choice of $G$,
\[9\gamma_{b,2}(G') \le \scalebox{1.4}{$\omega$}(G') = \scalebox{1.4}{$\omega$}(G)-10 < 9\gamma_{b,2}(G)-10 < 9(\gamma_{b,2}(G)-1) \]
and so $\gamma_{b,2}(G') +1 < \gamma_{b,2}(G)$.
Now we take a $2$-LD broadcast $f$ in $G'$ such that  $f(V(G'))=\gamma_{b,2}(G')$.
For every $x\in \{ v_1,v_2,u_1\}$, 
let $h_x(\nu)$ be a function on $V(G)$ defined by $1$ for $\nu=x$ and by $0$ for the others. 
If $f\cup h_{v_1}$ is a $2$-LD broadcast of $G$, then it has cost $\gamma_{b,2}(G')+1$, a contradiction.
Thus $f \cup h_{v_1}$ is not a $2$-LD broadcast of $G$, and so  
$f$ is not a $2$-LD broadcast of $G'-w_1u_2$.
It follows that
        exactly one vertex, say $y$, of $w_1$ and $u_2$
        hears from some vertex  under $f$ in $G'-w_1u_2$ and the other vertex, say $y'$, does not. 
        If $f(y)=2$, then
        let $x^*=y'$.
        If $f(y)\ne 2$, then
    letting $x^*\in\{u_1,v_2\}$ be the neighbor of $y'$.
    In each case, $f \cup h_{x^*}$ becomes a $2$-LD broadcast of $G$ with cost 
    $\gamma_{b,2}(G')+1$, a contradiction.
\end{proof}

\subsection{Refining possible cases}\label{subsec:allcases}

Throughout the remainder of the paper, we use the following notation and symbols, which will be  useful for the proofs. Some of them may seem redundant, but we introduce this notation to distinguish these arguments from earlier ones.

Let $t$ be a positive integer. 
For a vertex $v$ in $G$, let $p_t(v)$ denote the number of $t$-neighbors of $v$.
Recall that  $\delta(G)=2$ by Lemma~\ref{Lem:key3.1}(iii).
Hence $d(v)=p_2(v)+p_3(v)$.
For $X\subset V(G)$, let $A(X)$ be the set of $C_4$-components of $G-X$, and $I(X)$ be the set of isolated vertices of $G-X$.
We let $a(X)=|A(X)|$  and $i(X)=|I(X)|$.
Let $a_t(X)$ be the number of elements $C\in A(X)$ such that $|\partial(C)|=t$, and let $i_t(X)$ be the number of elements $w\in I(X)$ such that $d(w)=t$.
Note that  by Lemmas~\ref{Lem:removal-two-edges}  and~\ref{Lem:key3.1}(iii), 
$a(X)=a_3(X)+a_4(X)$  and  $i(X)=i_2(X)+i_3(X)$.
When there is no risk of confusion, we write $p_t$, $a$, $i$, $a_t$, and $i_t$ for $p_t(v)$, $a(X)$, $i(X)$, $a_t(X)$, and $i_t(X)$, respectively.
Note that $a(X)=b_1(G-X)$ and $i(X)=n_0(G-X)$.

Let $V^t$ be the set of $3$-vertices that are not on a triangle. For a vertex $v$ in $V^t$, let $B(v):=N_2[v]\setminus N[v]$, $\ell(v):=|E(G[B(v)])|$. 
For $t\in\{2,3\}$, let $\beta_t(v)$ be the number of $t$-vertices in $B(v)$. 
Let $\beta(v)=\beta_2(v)+\beta_3(v)$. 
If there is no confusion, we write $B$, $\ell$, $\beta_t$, and $\beta$ for $B(v)$, $\ell(v)$, $\beta_t(v)$, and $\beta(v)$, respectively.
See Figure~\ref{fig:structure by N_2[v]} for an illustration. From definition, it immediately holds that $p_2+p_3=3$ and $\beta\le {p_2+2p_3=3+p_3}$.
By Lemma~\ref{Lem:key3.1}(iii), 
$\beta_2+p_2\le 3$. 

\begin{figure}[h!]\centering
\includegraphics[page=7, width=8cm]{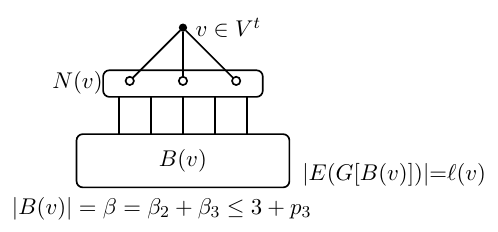}
\caption{The structure induced by $v$, $N(v)$, and $B(v)$}\label{fig:structure by N_2[v]}
\end{figure}

Let $V^*$ be the set of $3$-vertices $v$ such that $v\in V^t$ and
$\beta(v) \le 5$, that is,
\[ V^*=\{ v\mid v\text{ is not on a triangle, }  d(v)=3, \beta(v) \le 5 \}. \]
Note that $V^*\subset V^t$, and every $3$-vertex with a $2$-neighbor is in $V^*$.

In the remaining part of this subsection, we aim to show the following Lemma~\ref{Lem:possible:cases}, which  plays a key role in Subsections \ref{subsect:3.3} and \ref{subsect:3.4}. The lemma shows that the structure of $G - N_2[v]$ can occur in only three cases. We say \textit{$v$ is in the case $(Q_{r,a,i})$} if $a=a(N_2[v])$, $i=i(N_2[v])$, and $r=\gamma_{b,2}(G[\mathcal{C}(N_2[v])])$. 

\begin{Lem}\label{Lem:possible:cases}
Let $v\in V^*$, 
{$X=N_2[v]$,} and $v$ be in the case $(Q_{r,a,i})$. One of the following holds.
\begin{itemize}
\item[]  $[\ (Q_{3,0,1}) \ ]$ \quad  $r=3$, $a=0$, $i=1$, $\beta_2+\ell\le 1$,   and if $\beta_2=\ell=0$, then $|\partial(X)|\ge 2i_2+3i_3+2$.
\item[]  $[\ (Q_{4,0,2})\ ]$ \quad $r=4$, $a=0$, $i=i_2=2$, $\beta_2+\ell\le 2$, $\beta=5$, {and} $\beta_2\le 1$. 
\item[]  $[\ (Q_{4,1,0})\ ]$ \quad $r=4$, $a=1$, $i=0$, {and} $\beta_2=\ell=0$.
\end{itemize}
\end{Lem}

First, we collect some properties on the parameters defined previously.

\begin{Lem}\label{Lem:preview}
Let $v\in V^t$ and $X=N_2[v]$. 
Then the following hold.
\begin{itemize}
\item[\rm(i)] $ 3a+2i\le 3a_3+4a_4+2i_2+3i_3\leq |\partial(X)| \le 2\beta-\beta_2-2\ell $
\item[\rm(ii)]   
$\omega_G(\mathcal{C}(X))-|\partial(\mathcal{C}(X))|= 16a+6i+2(\beta_2+\ell)+18$.
\item[\rm(iii)]   If $|\partial(\mathcal{C}(X) )|\le 1$, then 
 $7a+6i+2(\beta_2+\ell)\equiv 7,8 \pmod{9}$. 
\end{itemize}  
\end{Lem}

\begin{proof}(i):  Note  that $p_2+p_3=3$.  
Since $v$ is not on a triangle by the assumption that $v\in V^t$, 
\[ |\partial(X)|= 2\beta_2+3\beta_3- (2p_2+3p_3-3+2\ell)=2\beta_2+3\beta_3- (3+p_3+2\ell).\]
Since $\beta\le {p_2+2p_3=3+p_3}$, we have
\[ 3a+2i\le 3a_3+4a_4+2i_2+3i_3 \leq |\partial(X)| \le 2\beta_2+3\beta_3-\beta-2\ell=2\beta-\beta_2-2\ell.  \]
where the second inequality holds by considering the $C_4$-components or the isolated vertices of $G-X$.
Thus (i) holds.

\noindent (ii): 
By Lemma~\ref{Lem:removal-two-edges}, the number of $2$-vertices contained in a $C_4$-component in $G-X$ is $a_3$, and so 
\begin{align*}
\omega_G(\mathcal{C}(X))
&= 3\times (\text{the number of $3$-vertices in }\mathcal{C}(X)) +4 \times(\text{the number of $2$-vertices in }\mathcal{C}(X))\\
&=  3(1+p_3+\beta_3+4a_4+3a_3+i_3)+4(p_2+\beta_2+i_2+a_3)\\
&=3\beta_3 +4\beta_2+12a_4+ 13a_3+4i_2+3i_3+p_2+12.
\end{align*}
In addition,
\[
|\partial(\mathcal{C}(X))|=|\partial(X)|-3a_3-4a_4 -2i_2-3i_3=  2\beta_2+3\beta_3- (3+p_3+2\ell)-3a_3-4a_4 -2i_2-3i_3.
\]
Therefore (ii) holds, since 
\begin{align*}
\omega_G(\mathcal{C}(X))-|\partial(\mathcal{C}(X))| 
&=(3\beta_3 +4\beta_2+12a_4+13a_3  +4i_2+3i_3+p_2+12) \\
&\qquad -( 2\beta_2+3\beta_3- (3+p_3+2\ell)-3a_3-4a_4 -2i_2-3i_3 )\\
&= 16a_4+16a_3+6(i_2+i_3)+2\beta_2+2\ell+18\\
& =16a+6i+2(\beta_2+\ell)+18.
\end{align*}

\noindent (iii):
If $\partial(\mathcal{C}(X) )=\emptyset$,
then $\mathcal{C}(X)=V(G)$ and so
 $\scalebox{1.4}{$\omega$}(G)= \omega_G(\mathcal{C}(X))=16a+6i+2(\beta_2+\ell)+18$ by (ii), which implies from Lemma~\ref{Lem:existence_deg_1_2} that 
\[ \scalebox{1.4}{$\omega$}(G)\equiv 7a+6i+2(\beta_2+\ell)\equiv 7,8 \pmod{9}.\]
If $|\partial(\mathcal{C}(X) )|=1$, then $\omega_G(\mathcal{C}(X))= 16a+6i+2(\beta_2+\ell)+19$ by (ii) and so 
$\scalebox{1.4}{$\omega$}( G[\mathcal{C}(X)])= \omega_G(\mathcal{C}(X))+1  =16a+6i+2(\beta_2+\ell)+20$, which implies from Lemma~\ref{Lem:cutedge4} that
\[  \scalebox{1.4}{$\omega$}( G[\mathcal{C}(X)])\equiv 7a+6i+2(\beta_2+\ell)+2\equiv 0,1 \pmod{9}.   \]
Hence (iii) holds.
\end{proof}

\begin{Lem}\label{Lem:every}
Let $v\in V^t$.
Then the following hold.
\begin{itemize}
    \item[\rm(i)]
For $X=N_2[v]$, 
    $a(X)>0$ or $i(X)>0$.
    \item[\rm(ii)] If $B(v)$ contains a $2$-vertex $x$ and $Y=N_2[v]\setminus \{u\}$ where $u$ is a vertex in $B(v)$ with exactly one neighbor in $X$ and $u\neq x$, then   $a(Y)>0$ or $i(Y)>0$. 
\end{itemize}
\end{Lem}

\begin{proof} 
By Lemma~\ref{Lem:key3.1}(i), $\gamma_{b,2}(G)\ge 4$ and so $V(G)\neq X$.   
Note that
$\gamma_{b,2}(G[X])\le 2$ and $\omega_G(X)\ge 3|X|+p_2= 3(\beta+4)+p_2$.  
Since  $v$ is not on a triangle, we have 
\begin{eqnarray}\label{eq:}
&&6-p_2 \le  |\partial(N[v])| \le  3\beta-\beta_2-|\partial(X)|.
\end{eqnarray}

\noindent (i): Suppose that $a(X)=i(X)=0$. Then by Fact~\ref{Fact:main_inequality}, 
$ |\partial(X)| >3(\beta+4)+p_2-9\gamma_{b,2}(G[X])\ge 3\beta+p_2-6$ and so $6-p_2> 3\beta-|\partial(X)|$, which is a contradiction to~\eqref{eq:}. Thus (i) holds.

{\noindent (ii):} Note that $\gamma_{b,2}(G[Y])\le 2$ 
and  $|\partial(Y)|= |\partial(X)|-(d(u)-2)$.
Suppose $a(Y)=i(Y)=0$.  
Since $x \in B$, we have $\beta_2\ge (3-d(u))+(3-d(x))=4-d(u)$  and $\omega_G(Y)\ge 3(\beta+3)+p_2+1$.
By Fact~\ref{Fact:main_inequality},  $|\partial(Y)| >3(\beta+3)+p_2+1-9\gamma_{b,2}(G[Y])\ge 3\beta+p_2-8$, and so 
\[  |\partial(X)|-(d(u)-2) > 3\beta+p_2-8 \ge -2 +\beta_2 +  |\partial(X)|\ge-2 + (4-d(u))+  |\partial(X)|\]
where the  second inequality is from~\eqref{eq:} and the last inequality is from the fact that $\beta_2\ge  4-d(u)$, a contradiction.
Thus (ii) holds.
\end{proof}

\begin{Lem}\label{Lem:C_adjacent_three-ver2}
Let $v\in V^t$. 
If $C$ is a $C_4$-component of $G-N_2[v]$  with $|\partial(C)|=3$, then 
$C$ is adjacent to distinct three vertices of $B(v)$.
\end{Lem}
\begin{proof}
Suppose to the contrary that there exists a $C_4$-component $C$ of $G-N_2[v]$ such that $|\partial(C)|=3$ and $C$ is adjacent at most two vertices of $B$. 
By Lemma~\ref{Lem:removal-two-edges}, $C$ is adjacent to exactly two vertices $u$ and $u'$ of $B$. 
We may assume that $u$ is adjacent to two vertices $w$ and $w'$ of $C$.
Let $w''$ be the $2$-vertex on $C$, and $w^*$ be the neighbor of $u'$ on $C$.
By Lemma~\ref{Lem:key3.1}(ii),
$u$ is not on any triangle and so
$w$ and $w'$ are not adjacent. 
We can check that $w''\in B(u)$ and $v$ is a vertex in $B(u)$ with exactly one neighbor in $N_2[u]$.
We let $v_1$ be the vertex in $N(v)$ adjacent to $u$.
If $v_1u'\in E(G)$, (see the first figure of Figure~\ref{fig:clm:C_adjacent_three- ver2})
then for the set $S=N_2[u]\setminus\{v\}$, $|\partial(S)|\leq 2$ and $\scalebox{1.4}{$\omega$}(G[S])=24$, which contradicts Lemma~\ref{Lem:cutedge4}.  
Thus $v_1u'\notin E(G)$. 
Let $v_2$ be a common neighbor of $u'$ and $v$.
Then $v_2\in B(w^*)$.
We let $Y=N_2[w^*]\setminus \{v_2\}
$.

Suppose that $v_2$ is not on a triangle.
Then $v_2$ has exactly one neighbor that is $u'$ in $N_2[w^*]$.
By the fact $\delta(G)\geq 2$, we can check $i(Y)=0$.
Suppose that $a(Y)\neq 0$. 
By Lemma~\ref{Lem:removal-two-edges}, $u'v_3\in E(G)$ (see the second figure of Figure~\ref{fig:clm:C_adjacent_three- ver2}) and there exists a vertex $x$ such that $vv_1xv_2v$ forms a $4$-cycle. 
By the symmetry of $v_2$ and $v_3$, it follows that  $v_3x\in E(G)$.
Then $G$ is determined with $\gamma_{b,2}(G)\leq 3$ by assigning $2$ to $u$ and $1$ to $u'$, which contradicts Lemma~\ref{Lem:key3.1}(i). Thus $a(Y)=0$, which contradicts Lemma~\ref{Lem:every}(ii).

\begin{figure}\centering
\includegraphics[page=8, width=13.5cm]{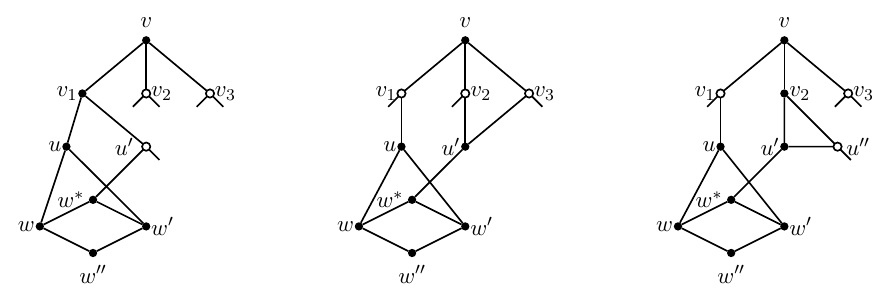}
\caption{Illustration of graphs in Lemma~\ref{Lem:C_adjacent_three-ver2}} \label{fig:clm:C_adjacent_three- ver2}
\end{figure}

Now suppose that $v_2$ is on a triangle.
We let $u''$ be the common neighbor of $v_2$ and $u'$ and see the third figure of Figure~\ref{fig:clm:C_adjacent_three- ver2}. Then consider $Z=N_2[w]$. Since $P:u''v_2vv_3$ is a path in $G-Z$ and there are three edges between $P$ and $Z$.
If $a(Z)\neq 0$, then $vv_2u''v_3$ is a $4$-cycle and so $v_3$ is a $2$-vertex, which contradicts Lemma~\ref{Lem:key3.1}(ii).
Thus $a(Z)=0$. 
By the fact $\delta(G)\ge 2$, 
we can check $i(Z)=0$, which contradicts Lemma~\ref{Lem:every}(i).  
\end{proof}

\begin{Fact}\label{observation:r}
For a $3$-vertex $v$, let $X=N_2[v]$, and then 
\[
\gamma_{b,2}(G[\mathcal{C}(X)])  \le \min\{2a+i+2,a_3+p_3+3\}.
\]
Moreover, if there is a vertex $u \in B(v)$ such that $u$ is adjacent to two distinct $\zeta_1,\zeta_2\in I(X)\cup A(X)$, then
\[ \gamma_{b,2}(G[\mathcal{C}(X)])  \le 2a+i+1.\]
\end{Fact}

\begin{proof}
For simplicity, let $r=\gamma_{b,2}(G[\mathcal{C}(X)])$. By Lemma~\ref{Lem:removal-two-edges}, the number of $2$-vertices of $C_4$-components in $G-X$ is $a_3$. 
We have $r\le 2p_3+p_2+a_3 = 3+p_3+a_3$, by assigning $2$ to  each of the $3$-neighbors of $v$, $1$ to the neighbor of the $2$-neighbor of $v$ other than $v$, and each $2$-vertex on a $C_4$-component in $G-X$. 
On the other hand, $r \le  2+2a+i$ 
by assigning $2$ to $v$ and one vertex of each $C_4$-component $C$ of $G-X$, $1$ to each isolated vertex in $G-X$.

To show the `moreover' part, 
let $f$ be a function defined on $\mathcal{C}(X)$ by assigning $2$ to each of $v$ and one vertex on each $C_4$-component of $G-X$ other than $\zeta_1,\zeta_2$ and $1$ to all isolated vertices of $G-X$ other than $\zeta_1,\zeta_2$, and $0$ to the other vertices. 

If $\zeta_1$ and $\zeta_2$ are isolated vertices of $G-X$, then $r\leq 2+2a+(i-2)+1=2a+i+1$ from $f$ by assigning $1$ to $u$. Suppose that $\zeta_1$ is a $C_4$-component of $G$.
Let $x$ be a vertex on $\zeta_1$ adjacent to $u$. 
If $\zeta_2$ is an isolated vertex of $G-X$, then $r\le {\rm cost}(f) +2 = 2+2(a-1)+(i-1)+2=2a+i+1$ from $f$ by assigning $2$ to $x$. If $\zeta_2$ is a $C_4$-component, 
then $r\le {\rm cost}(f)+2+1= 2+2(a-2)+i+2+1=2a+i+1$ from $f$ by assigning $2$ to $x$ and $1$ to $y$, where $y$ is the vertex on $\zeta_2$ that is not adjacent to $N_2[x]$.
\end{proof}

The following remark follows immediately from the proof of the
above fact. 

\begin{Rmk}\label{rmk:r}
By the same reason of the `moreover' part of Fact~\ref{observation:r},
if there are two vertices $u_1$ and $u_2$ in $B$ such that 
$u_1$  is adjacent to $\zeta_1,\zeta_2$ and $u_2$  is adjacent to $\zeta_3,\zeta_4$
for some distinct $\zeta_1,\zeta_2,\zeta_3,\zeta_4\in I(X)\cup A(X)$,  $ \gamma_{b,2}(G[\mathcal{C}(X)])  \le 2+2a+i-2=2a+i$.
\end{Rmk}

Now we are ready to prove Lemma~\ref{Lem:possible:cases}.

\begin{proof}[Proof of Lemma~\ref{Lem:possible:cases}]  Note that $V^*\subset V^t$.  By Lemma~\ref{Lem:preview}(i), 
\begin{equation}\label{eq:Lem:possible:cases:3} 3a+2i\le 3a_3+4a_4+2i_2+3i_3\leq |\partial(X)| \le   10-\beta_2 -2\ell \le 10
\end{equation}
where the second inequality holds by considering the $C_4$-components or the isolated vertices of $G-X$.
By Lemma~\ref{Lem:every}, $a+i>0$ and so 
 $\mathcal{C}(X)\neq X$.  
Recall that  $\beta_2+p_2\le 3$.

\begin{Clm}\label{claim:r_inequality} 
The following hold:
\begin{itemize}
\item[\rm(i)]  $a\le 2$, $i_2\le 2$, $i\le 3$,   $\beta_2+2i_2\le \beta$, $a+i\le 4$.  
Moreover, if $a+i=4$, then $\beta_2=\ell=0$, $|\partial(X)|=10$, $\beta=5$, $a=a_3\in \{1,2\}$, $i_2=2$, $p_2=1$, and $V(G)=\mathcal{C}(X)$.
\item[\rm(ii)] $r\le 7$.
\item[\rm(iii)] $9r> 16a+6i+2(\beta_2+\ell)+18.$
\end{itemize}
\end{Clm}
\begin{proof}
\noindent  (i): By \eqref{eq:Lem:possible:cases:3}, we have $3a\le 10$ and so $a\le 3$.
If $a=3$, then it follows from $3a+2i\le 10-\beta_2-2\ell$ that $i=0$ and $\beta_2\le 1$, $\ell=0$, and 
so $|\partial(\mathcal{C}(X) )|\le 1$, which contradicts Lemma~\ref{Lem:preview}(iii). Thus $a\le 2$. 

Since $\beta \leq 5$, $i_2\leq 2$ by Lemma~\ref{Lem:key3.1}(iii).
By \eqref{eq:Lem:possible:cases:3},  $2i_2+3i_3\leq 10$, and so $i\le 4$.
If $i=4$, then $i_2=i_3=2$ and $a=\beta_2=\ell=0$  by \eqref{eq:Lem:possible:cases:3}, and so $|\partial(\mathcal{C}(X) )|\le 1$, which contradicts Lemma~\ref{Lem:preview}(iii). Thus $i\le 3$. 

By Lemma~\ref{Lem:key3.1}(iii), it holds that $\beta_2+2i_2\le \beta$. Since $i_2\le 2$, it holds from \eqref{eq:Lem:possible:cases:3} that
\[ 3a+3i-2\le 3a+3i-i_2 = 3a+2i_2+3i_3 \le |\partial(X)|\le 10-\beta_2-2\ell.\]
Then $3a+3i\le 12$ and so $a+i\le 4$. 

Now suppose $a+i=4$.
Then we can check that $\beta_2=\ell=0$, $|\partial(X)|=10$, $\beta=5$, $a=a_3$, $i_2=2$, $p_2=1$, and $G=G[\mathcal{C}(X)]$ by \eqref{eq:Lem:possible:cases:3}.
Thus, Lemma~\ref{Lem:preview}(iii)  implies that 
$a+24\equiv 7,8 \pmod{9}$. Then $a\in \{1,2\}$ and so the  `moreover' part also holds. Thus (i) holds.

\noindent (ii): By Fact~\ref{observation:r},  $r\le \min\{2a+i+2, a_3+p_3+3\}\le a+(a+i)+2\le 2+4+2=8$ where the last inequality holds by (i). 
Suppose to the contrary that $r=8$. Then  $a=2$ and $a+i=4$.   
By the `moreover' part of (i),   $a=a_3$
and $p_2=1$ and so $p_3=2$. Then $r\leq a_3+p_3+3=7$ by Fact~\ref{observation:r}, a contradiction.   
 
\noindent (iii):  By Fact~\ref{Fact:main_inequality},  $9\gamma_{b,2}(G[\mathcal{C}(X)]) >  \omega_G(\mathcal{C}(X))-|\partial(\mathcal{C}(X))|$, and equivalently,
by Lemma~\ref{Lem:preview}(ii), $9r > 16a+6i+2(\beta_2+\ell)+18.$ Thus (iii) holds. 
\end{proof}
We have $r\ge 3$ by Claim~\ref{claim:r_inequality}(iii).
All cases in which $a, r, \beta_2, \ell$ satisfy  Fact~\ref{observation:r} and Claim~\ref{claim:r_inequality} are listed  in Table~\ref{table1}.
  \begin{table}[h!]
\centering\footnotesize{  
\begin{tabular}{c||c|c|c|c|c}
 $(a,i)$    & $r=3$     & $r=4$ & $r=5$ & $r=6$ & $r=7$ \\ \hline  \hline
 $(0,1)$ &  $\beta_2+\ell \le 1 $ $(\S)$  &  \cellcolor{red!6}   &   \cellcolor{red!6}    &   \cellcolor{red!6}   &   \cellcolor{red!6}          \\   
 $(0,2)$ &  \cellcolor{blue!6}    &  $\beta_2+\ell \le  2$ $(\S)$   &  \cellcolor{red!6}      &   \cellcolor{red!6}   &   \cellcolor{red!6}       \\
 $(0,3)$ &   \cellcolor{blue!6} & \cellcolor{blue!6}   &   $\beta_2+\ell \le 4$   &  \cellcolor{red!6}    &   \cellcolor{red!6}        \\  
 $(1,0)$ &   \cellcolor{blue!6}  &   $\beta_2=\ell=0$   &  \cellcolor{red!6}      &   \cellcolor{red!6}      &   \cellcolor{red!6}      \\
 $(1,1)$ &  \cellcolor{blue!6}   &  \cellcolor{blue!6}  &    $\beta_2+\ell\le 2$   &   \cellcolor{red!6}    &   \cellcolor{red!6}      \\
 $(1,2)$ &  \cellcolor{blue!6}  &   \cellcolor{blue!6} &  \cellcolor{blue!6}    &    $\beta_2+\ell\le 3$  &   \cellcolor{red!6}       \\ 
  $(1,3)$ & \cellcolor{blue!6}    &  \cellcolor{blue!6}  &    \cellcolor{blue!6}  &    $\beta_2=\ell=0$    &  $\beta_2+\ell\le 5$         \\   
 $(2,0)$  &   \cellcolor{blue!6}  &  \cellcolor{blue!6}  &    \cellcolor{blue!6} &  $\beta_2+\ell\le 1$    &  \cellcolor{red!6}         \\
 $(2,1)$   & \cellcolor{blue!6}    &  \cellcolor{blue!6}  & \cellcolor{blue!6}   &    \cellcolor{blue!6}    &   $\beta_2+\ell\le 3$    \\   
 $(2,2)$   &  \cellcolor{blue!6}   &   \cellcolor{blue!6} &    \cellcolor{blue!6}  &    \cellcolor{blue!6}    &   $\beta_2=\ell=0$     \\   
\end{tabular}}
\caption{The cells shaded in red correspond to the cases that do not satisfy $r \le 2 + 2a + i$, and they are impossible by Fact~\ref{observation:r}, and the cells shaded in blue correspond to the cases that do not satisfy $9r> 16a+6i+2(\beta_2+\ell)+18$ and so they are impossible by Claim~\ref{claim:r_inequality}(iii). 
The conditions of the cases corresponding to the unshaded cells  are obtained by Claim~\ref{claim:r_inequality}(iii).} 
\label{table1}
\end{table}

First, we will show that $r\le 4$, that is, $v$ is not in the following cases: 
\[ (Q_{6,1,3})   \quad    (Q_{7,1,3}) \quad 
(Q_{7,2,2}) \quad    (Q_{5,0,3})  \quad(Q_{6,1,2})  
 \quad (Q_{7,2,1})   \quad (Q_{6,2,0})\quad (Q_{5,1,1}).  \]

\begin{Clm}\label{clm:a+i<=2}It holds that $a+i\le 2$. 
Thus, $v$ cannot be in one of six cases $(Q_{6,1,3})$, $(Q_{7,1,3})$, $(Q_{7,2,2})$, $(Q_{5,0,3})$, $(Q_{6,1,2})$, and $(Q_{7,2,1})$.
\end{Clm}
\begin{proof}
To the contrary, suppose $a+i>2$.
Then by Claim~\ref{claim:r_inequality}(i) $a+i=3$ or $4$. 
Suppose that $a+i=4$. 
Thus it is one of the cases $(Q_{6,1,3})$, $(Q_{7,1,3})$ and $(Q_{7,2,2})$.
By the `moreover' part of Claim~\ref{claim:r_inequality}(i),
$\beta_2=\ell=0$, $|\partial(X)|=10$, $\beta=5$, $a=a_3\in \{1,2\}$, $i_2=2$, $p_2=1$, and $V(G)=\mathcal{C}(X)$.
Let 
$u_1$ be the neighbor of the $2$-neighbor of $v$ with $u_1\neq v$ and we also let $w_1$ and $w_2$ be the $2$-vertices that are isolated in $G-X$.
By Lemma~\ref{Lem:key3.1}(iii),
$N(w_1)\cup N(w_2)=B\setminus\{u_1\}$ and $N(w_1)\cap N(w_2)=\emptyset$.
We let $N(w_1)=\{u_2,u_3\}$ and $N(w_2)=\{u_4,u_5\}$.

We consider the case $(Q_{6,1,3})$ or $(Q_{7,1,3})$. 
Let $C$ be the $C_4$-component in $G-X$.
Since $i_2=2$, $i_3=1$.
Let $w$ be the $3$-vertex that is isolated vertex in $G-X$.
Since $N(w_1)\cup N(w_2)=B\setminus\{u_1\}$ and $|E(B\setminus\{u_1\}, \mathcal{C}(X)\setminus X)|\le 8$, $u_1$ is adjacent to $C$. If $u_1w\in E(G)$, then $r\leq 5$ by assigning $1$ to $v,w_1,w_2,u_1,x$ where $x$ is a vertex on $C$ properly, a contradiction. 
Thus $u_1w\notin E(G)$, and so we may let $N(w)=\{u_2,u_3,u_4\}$.
Then $u_5$ is adjacent to a vertex $z$ in $C$.
Hence $r\leq 2a+i=5$ by Remark~\ref{rmk:r}, a contradiction.

We consider the case $(Q_{7,2,2})$. 
Let $C$ and $C'$ be the $C_4$-components in $G-X$.
Since $a=a_3$, $|\partial(C)|=|\partial(C')|= 3$. 
Thus, we may assume that $u_2$ is adjacent to $V(C)$ and 
$u_4$ is adjacent to $V(C')$.
Then $r\leq 2a+i=6$ by Remark~\ref{rmk:r}, a contradiction.
Thus the case $(Q_{7,2,2})$ is  impossible.

It remains to show that  $(Q_{5,0,3})$, $(Q_{6,1,2})$, $(Q_{7,2,1})$ are impossible. 
Since $a+i\ge3$, and so $2a+2i\ge 6 > 5 \geq|B|$ and so there exists a vertex in $B$ adjacent to two elements $\zeta_1$ and $\zeta_2$, where  $\zeta_1,\zeta_2\in I(X)\cup A(X)$. Thus the cases $(Q_{5,0,3})$, $(Q_{6,1,2})$, $(Q_{7,2,1})$ are impossible, by the `moreover' part of Fact~\ref{observation:r}.
\end{proof}

To complete the proof to show that $r\le 4$, it remains to show that the two cases $(Q_{6,2,0})$, $(Q_{5,1,1})$ 
are impossible by Claim~\ref{clm:a+i<=2}. Suppose that $v$ is in the case $(Q_{6,2,0})$ or  $(Q_{5,1,1})$. 
Suppose that there exists a vertex in $B$ adjacent to two elements $\zeta_1$ and $\zeta_2$, where  $\zeta_1,\zeta_2\in I(X)\cup A(X)$.
 Thus, by the `moreover' part of Fact~\ref{observation:r},  we have $r\le 2a+i+1$. Therefore, the two cases are impossible. 
Hence each vertex in $B$ is adjacent to at most one $\zeta\in I(X)\cup A(X)$.
By Lemma~\ref{Lem:preview}(i), $|\partial(X)|\ge 5$ and so $\beta \ge 3$.
Note that $|V(G)|\ge   |N[v]|+\beta + 4a+i \ge  4+\beta+5\ge 12$.
Let $C$ be a $C_4$-component of $G-X$.  

\begin{Clm}\label{clm:C_adjacent_three}
$C$ is adjacent to at least three vertices of $B$. 
\end{Clm}
\begin{proof}
Suppose to the contrary that $C$ is not adjacent to at least three vertices of $B$. 
By Lemma~\ref{Lem:removal-two-edges}, $C$ is adjacent to only two vertices $u$ and $u'$ of $B$. Then $|\partial(C)|=4$ by Lemma~\ref{Lem:C_adjacent_three-ver2}. 
Let $v_1$ be a neighbor of $u$ not in $V(C)$.  
Then $v_1$ is a common neighbor of $u$ and $v$.
Now we consider $Z=V(C)\cup\{u,u',v_1\}$. Then $Z=N_2[z]$ for some vertex $z$ in $C$  and  $\omega_G(Z)=18+6-d(v_1)\geq21$, $|\partial(Z)|\le 3$, and $\gamma_{b,2}(G[Z])\leq 2$ by assigning $2$ to $z$.
By Fact~\ref{Fact:main_inequality},
$i(Z)>0 $ or $a(Z)>0$.
If $G-Z$ has a $C_4$-component $C'$, then by Lemma~\ref{Lem:removal-two-edges}, $V(G)=Z\cup V(C')$ and so $|V(G)|=11$, a contradiction to the fact that $|V(G)|\ge 12$. 
Thus $a(Z)=0$ and so $i(Z)>0$. Since $\delta(G)\geq 2$, every isolated vertex in $G-Z$ is adjacent to $u'$ and has degree $2$ in $G$, so  $u'\notin N_2[v]$, a contradiction.
\end{proof}

Recall that each vertex in $B$ is adjacent to at most one element $\zeta$  in $I(X)\cup A(X)$  and  $\beta\le 5$.
By Claim~\ref{clm:C_adjacent_three},
the case $(Q_{6,2,0})$ does not happen and so $v$ is in the case $(Q_{5,1,1})$.
Let $w_0$ be the isolated vertex in $G-X$.
Then for some vertex $z$ of $C$,
$N_2(z)\cup N(w_0)$ contains $B$ by Claim~\ref{clm:C_adjacent_three}. 
Thus $r\le 4$, by assigning $1$ to each of $v$ and $w_0$ and $2$ to $z$, a contradiction. 

By comparing the cases with our final goal, it remains to consider the cases $(Q_{3,0,1})$ or $(Q_{4,0,2})$ corresponding to the cells with  the mark ($\S$).
We first consider the case $(Q_{3,0,1})$.
Suppose $\beta_2=\ell=0$ and $|\partial(X)|<2i_2+3i_3+2$.
Then $|\partial(X)|\leq 2i_2+3i_3+1$ and so $|\partial(\mathcal{C}(X))|\leq 1$, which contradicts Lemma~\ref{Lem:preview}(iii). 

Now we consider the case $(Q_{4,0,2})$.
As $r>2a+i+1$, by the `moreover' part of Fact~\ref{observation:r},  each vertex in $B$ is adjacent to at most one element $\zeta$ in $I(X)\cup A(X)$.
Thus $2i_2+3i_3  \le \beta$. Since $\beta\leq 5$, $i_3\leq 1$.
Then $i_2\geq 1$.
Let $w_1$ and $w_2$ be the isolated vertices of $G-X$.
If $N(w_1)\cup N(w_2)$ contains $B$, then $r\le 3$ by assigning $1$ to $w_1$, $w_2$, and the vertex $v$, a contradiction. 
Thus $N(w_1)\cup N(w_2)$ does not contain $B$, which implies that $i_2=2$ and $\beta=5$.  By Lemma~\ref{Lem:key3.1}(iii), each $2$-vertex that is an isolated vertex in $G-X$ is not adjacent to a $2$-vertex in $B$. 
Thus  $4= 2i_2\le \beta-\beta_2 $, and so $4+\beta_2 \le \beta$.  
Hence $\beta_2\leq 1$.  
\end{proof}

We finish this subsection by giving a further analysis which will be used later.

\begin{Lem}
    \label{Lem:two-sets}
    Let $z_1$ and $z_2$ be two distinct vertices such that $N[z_1]\cap N[z_2]=\emptyset$, and let $Z=N[z_1]\cup N[z_2]$. If $Z\setminus \{z_1,z_2\}$ has $t$ $2$-vertices and $|E(N[z_1],N[z_2])|=s$, then \[3i(Z)+2a(Z)>2s+2t-6.\] 
\end{Lem}
\begin{proof}
    Suppose that $Z\setminus \{z_1,z_2\}$ has $t$ $2$-vertices and $|E(N[v_1],N[v_2])|=s$.
        Since $N[z_1] \cap N[z_2]=\emptyset$, \[\omega_G(Z)\ge 3(d(z_1)+1) +3(d(z_2)+1) +  (3-d(z_1))+(3-d(z_2)) + t=2(d(z_1)+d(z_2))+12+t.\] 
        In addition,
    $|\partial(Z)|\le 2 (d(z_1)+d(z_2) )-2s -t$. 
    Obviously, $\gamma_{b,2}(G[Z])\leq 2$.
By Fact~\ref{Fact:main_inequality}, $3i(Z)+2a(Z)>\omega_G(Z)-|\partial(Z)| -9 \gamma_{b,2}(G[Z]) $, and thus 
 \[       3i(Z)+2a(Z)>  2(d(z_1)+d(z_2))+12+t -  ( 2 (d(z_1)+d(z_2) )-2s -t )-18 = 2s+2t-6.\qedhere \]
\end{proof}

\begin{Lem}\label{Lem:no_301}
Let $v\in V^{*}$ and $X=N_{2}[v]$. 
Suppose that $v$ has a $2$-neighbor and  $v$ is in the case $(Q_{3,0,1})$ in Lemma~\ref{Lem:possible:cases} and let $w$ be the isolated vertex in $G-X$. There is a vertex $u\in B(v)\setminus N(w)$ such that $N(u)\subset N(v)\cup N(w)$, and $|N(u)\cap N(v)|\geq 2$.
\end{Lem}

\begin{proof}
Note that $\beta_2+\ell \le 1$.
Let $N(v)=\{v_1,v_2,v_3\}$ and $v_1$ be a $2$-vertex. 
Let $q=\scalebox{1.4}{$\omega$}(G-\mathcal{C}(X))-9\gamma_{b,2}(G-\mathcal{C}(X))$.  
By Lemma~\ref{Lem:preview}(ii) and  the `moreover' part of  Fact~\ref{Fact:main_inequality}, 
it holds that $9r-q>16a+6i+2(\beta_2+\ell)+18$ where $r=\gamma_{b,2}(G[\mathcal{C}(X)])$.
For the case ($Q_{3,0,1}$), $27-q>24+2(\beta_2+\ell)$ and so $q\le 2$. Moreover, if $q\ge 1$, then  $\beta_2=\ell=0$. Thus, noting that $G-\mathcal{C}(X)
$ has no isolated vertex, we can conclude that for each component $D$ of $G-\mathcal{C}(X)$, $D$ has no size one or three, and if $D$ has size two, then $q\ge1$ and so $\ell=0$.

Let $Z=N[v]\cup N[w]$, $|E(N[v],N[w])|=s$ and 
$Z\setminus\{v,w\}$ contain $t$ $2$-vertices.
Then $s\ge d(w)$ and $t\ge 1$. By Lemma~\ref{Lem:two-sets}, $3i(Z)+2a(Z)>2s+2t-6\ge 2d(w)+2-6\ge 0$. Thus $i(Z)>0$ or $a(Z)>0$.

Note that $|X\setminus Z|=\beta-d(w)$.
Since $G-\mathcal{C}(X)$ has neither an isolated vertex nor a $C_4$-component and $Z\setminus\{w\}$ is a subset of $X$, 
there is a vertex $u\in X-Z$ such that $u$ is an isolated vertex of $G-Z$ or a vertex on a $C_4$-component of $G-Z$. 

Suppose that there exists a $C_4$-component $C$ of $G-Z$ containing $u$. 
For simplicity, let $U=V(C)\cap (X-Z)$.
Then $U\neq \emptyset$.
Since $\ell\le 1$, $U$ has at most two vertices.
By the argument in the first paragraph, each component of $G-\mathcal{C}(X)$ has no size one or three. It follows that $|U|=2$. Then $\ell=1$, which is also a contradiction. Therefore $u$ is an isolated vertex of $G-Z$ and so $N(u)\subset N(v)\cup N(w)$.

Note $\beta_2+\ell\le 1$.
Since $u \in B$,
$u$ is adjacent to at most one vertex in $B(v)$.
Thus $|N(u) \cap N(w) |\le 1$.
If $d(u)=3$, then
$|N(u)\cap N(v)|\ge 2$.
If $d(u)=2$, then $\beta_2=1$ and so $\ell=0$, which implies $|N(u)\cap N(v)|= 2$ since $u$ is not adjacent to any vertex in $B$.
Thus in each case, $|N(u)\cap N(v)|\ge  2$.
\end{proof}

\subsection{The size of $N_2[v]$ when $v$ has a $2$-neighbor}\label{subsect:3.3}
 
Let $V^{**}$ be the set of $3$-vertices with a $2$-neighbor. Note that 
\[V^{**}\subset V^* \subset V^t.\] 
In this subsection, we aim to establish that 
$\beta(v) = 5$ for every vertex $v \in V^{**}$. 
As a preliminary step, we first prove that 
$\beta(v) \geq 4$ for all $v \in V^{**}$.
In the following proofs, as in the previous subsection, for a fixed vertex $v$, the notations $\beta$, $\beta_2$ $B$, $\ell$ are understood to be $\beta(v)$, $\beta_2(v)$, $B(v)$, and $\ell(v)$, respectively.
We also recall that every $4$-cycle of $G$ is an induced cycle by Lemma~\ref{Lem:b(G)=0}.

\begin{Lem}\label{Lem:no_intersection_induced cycle_length4}
Let $v\in V^{**}$. Then $\beta(v)\ge 4$.
Hence, for every $4$-cycle $C$ containing a $2$-vertex, there is no other $4$-cycle  sharing any edge in $G$ with $C$. 
\end{Lem}

\begin{proof}
Let $X=N_2[v]$ and $N(v)=\{v_1,v_2,v_3\}$ and we may assume that $v_1$ is a $2$-vertex and $v_2$ and $v_3$ are $3$-vertices by Lemma~\ref{Lem:key3.1}(iii).
Let $u_1$ be the neighbor of $v_1$ other than $v$.

To the contrary, suppose
$\beta<4$. It is clear that $\beta\geq 2$.
If $\beta=2$, then $|\partial(X)|=1$, which contradicts Lemma~\ref{Lem:cutedge4}.
Thus $\beta=3$.
Let $r=\gamma_{b,2}(G[\mathcal{C}(X)])$.
Since $v\in V^*$,
$v$ is in the case $(Q_{3,0,1})$ or $(Q_{4,1,0})$ by Lemma~\ref{Lem:possible:cases}.

\noindent{(Case 1)} Suppose that $v$ is in the case $(Q_{3,0,1})$. 
Let $w$ be the isolated vertex in $G-X$.
If $d(w)=3$, then $r\le 2$ by assigning $1$ to $v$ and $w$  since $\beta=3$, a contradiction.
By the fact $\delta(G)\ge 2$, $d(w)=2$.
By Lemma~\ref{Lem:key3.1}(iii),
we may let $N(w)=\{u_2,u_3\}$ where $u_2$ and $u_3$ are vertices in $B$ other than $u_1$ and $d(u_2)=d(u_3)=3$. 
By Lemma~\ref{Lem:key3.1}(iii)
$\beta_2=0$.
By Lemma~\ref{Lem:no_301}, there is a vertex $u\in B$ such that $N(u)\subset N(v)\cup N(w)$ and $|N(u)\cap N(v)|\ge 2$.
Thus $u=u_1$. 
Suppose that $\ell=1$. Since $u_2u_3\notin E(G)$ by Lemma~\ref{Lem:key3.1}(ii), 
it follows that $|N(u_1)\cap N(v)|= 2$ and $|N(u_1)\cap \{u_2,u_3\}|=1$.
Then $|\partial(X)|=d(u_1)+d(u_2)+d(u_3)-5-2\ell = 4-2\ell=2$ by counting the edges incident with $B$, which determines $G$ so that $|V(G)|=8$, a contradiction to Lemma~\ref{Lem:key3.1}(i).
Thus $\ell=0$, and so $N(u_1)=N(v)$. 
Consider $X'=N_2[v_2]\setminus\{z\}$, where $z$ is the vertex in $N_2[v_2]\setminus \mathcal{C}(X)$. 
Then $X'= \mathcal{C}(X)\setminus\{u_i\}$  for some $u_i\in\{u_2,u_3\}$ and so one can observe from $d_{G-X'}(u_i)=1$.
In addition, we can check that $z$ has only one neighbor in $X'$, $v_1 \in B(v_2)$, and by the fact $\delta(G)\ge 2$ and Lemma~\ref{Lem:removal-two-edges}, $a(X')=i(X')=0$, which contradicts Lemma~\ref{Lem:every}(ii).

\noindent{(Case 2)} Suppose that $v$ is in the case $(Q_{4,1,0})$. Then $\beta_2=\ell=0$. 
Since $\beta_2=0$, each vertex in $B$ is a $3$-vertex.
Then there are following three possible cases as shown in Figure~\ref{fig:Lem:no_intersection_induced cycle_length4_0}.
Note that $v_2,v_3\in V^t$. 
Let $C$ be the $C_4$-component in $G-X$.
If $|\partial(C)|=4$, then $V(G)=X\cup V(C)$, and then $i(N_2[v_2])=a(N_2[v_2])=0$, which is a contradiction to Lemma~\ref{Lem:every}(i). Hence $|\partial(C)|=3$ and $|\partial(\mathcal{C}(X))|=1$. 
By Lemma~\ref{Lem:C_adjacent_three-ver2}, $C$ is adjacent to each of $u_1$, $u_2$, and $u_3$.  Let $z_i$ be the vertex  on $C$ such that $u_iz_i\in E(G)$   for each $i\in \{1,2,3\}$. Thus, the second  figure of Figure~\ref{fig:Lem:no_intersection_induced cycle_length4_0} is impossible.

If it is the case of the first figure of Figure~\ref{fig:Lem:no_intersection_induced cycle_length4_0} and $z_2z_3\in E(G)$, then $a(N_2[v_2])=i(N_2[v_2])=0$, which is 
a contradiction to Lemma~\ref{Lem:every}(i).
If it is the case of the  third figure of Figure~\ref{fig:Lem:no_intersection_induced cycle_length4_0} and $z_1z_2\in E(G)$, then $a(N_2[v_2])=i(N_2[v_2])=0$, which is 
a contradiction to Lemma~\ref{Lem:every}(i).
Hence, in any case, $z_2$ has exactly one neighbor in $N_2[v_2]$.
Let $Y=N_2[v_2]\setminus \{z_2\}$.
We can check $v_1,z_2\in B(v_2)$. 
The component of $G-Y$ containing $u_1$  or $u_3$ contains at least $5$ vertices and so $a(Y)=i(Y)=0$, which contradicts Lemma~\ref{Lem:every}(ii).
\begin{figure}[h!]\centering
\includegraphics[page=9,width=13.5cm]{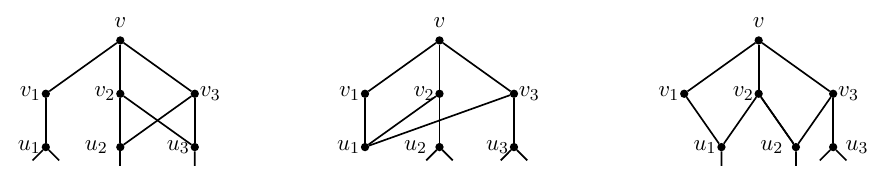}
 \caption{ Subgraphs mentioned in Lemma~\ref{Lem:no_intersection_induced cycle_length4}}
\label{fig:Lem:no_intersection_induced cycle_length4_0}
\end{figure}
\end{proof}
 
\begin{Lem} \label{Lem:two-vertex-removal_C_4}
The removal of at most two vertices of $G$ cannot create a $C_4$-component.
\end{Lem}
\begin{proof}
To the contrary, suppose that there are at most two vertices of $G$ whose removal creates a $C_4$-component. 
By Lemma~\ref{Lem:removal-two-edges},
there are exactly two vertices $u$ and  $v$ of $G$ whose removal creates a $C_4$-component $C$. We take such vertices $u$ and $v$ so that $|\partial(C)|$ is small as possible.
We may assume that $u$ is adjacent to exactly two vertices $x$ and $y$ on $C$. 
Then $d(u)=3$ by Lemmas~\ref{Lem:key3.1}(ii) and ~\ref{Lem:no_intersection_induced cycle_length4}.

If $|\partial(C)|=3$, then $x \in V^{**}$ or $y \in V^{**}$ and $\beta(x)\leq 3$ or $\beta(y)\leq 3$, which contradicts Lemma~\ref{Lem:no_intersection_induced cycle_length4}.
Thus $|\partial(C)|=4$.
Hence
$v$ is adjacent to exactly two vertices of $C$ and $d(v)=3$ by Lemma~\ref{Lem:no_intersection_induced cycle_length4}.
Let $u'$ and $v'$ be the neighbor of $u$ and $v$ outside $C$, respectively.
Let $Z=N[u]\cup N[v]$.
If $u'=v'$, then $|\partial(Z)|\le 1$ and $\scalebox{1.4}{$\omega$}(G[Z])=22$, which contradicts Lemma~\ref{Lem:cutedge4}.  
Thus $u'\neq v'$.
Let $X=N_2[x]$. Then $|X|=7$ and so $\omega_G(X)\ge 21$.
Since $|\partial(X)|\le3$, by the minimality of $|\partial(C)|$, $a(X)=0$. We note $\gamma_{b,2}(G[X])\leq 2$ and so 
by Fact~\ref{Fact:main_inequality}, $G-X$ has an isolated vertex $w$.
Since $\delta(G)\ge2$, $N(w)=\{v,u'\}$.
In addition, $|\partial(\mathcal{C}(X))|\le 1$,
$\omega_G(\mathcal{C}(X))\ge 24$. Note that $\gamma_{b,2}(G[\mathcal{C}(X)])\leq 2$ by assigning $1$ to both $u$ and $v$, which contradicts Fact~\ref{Fact:main_inequality}.
\end{proof} 

\begin{Lem}\label{Lem:no_inter_cycle_length4 and 5 6} A $4$-cycle containing a $2$-vertex does not share an edge with a cycle of length at most six. 
\end{Lem}
\begin{proof}
    Suppose that $G$ has a $4$-cycle $C$ containing a $2$-vertex such that $C$ shares an edge with some cycle $C'$ of length at most six.
    Then $C$ is an induced cycle by Lemma~\ref{Lem:b(G)=0}.
    We take such $C$ and $C'$ so that 
$|E(C')|$ is the smallest, with respect to $|E(C)\cap E(C')|$ is as small as possible.
Take a $3$-vertex $v$ adjacent to the $2$-vertex on $C$. Let $N(v)=\{v_1,v_2,v_3\}$ where $v_1$ is the $2$-vertex on $C$ and $v_3$ is a $3$-vertex not on $C$.
Let $u_1$ be the vertex nonadjacent to $v$ on $C$. By Lemma~\ref{Lem:key3.1}(iii), $d(v_2)=d(v_3)=d(u_1)=3$.
By Lemmas~\ref{Lem:key3.1}(ii) and~\ref{Lem:no_intersection_induced cycle_length4}, $C'$ is an induced cycle and has length $5$ or $6$ with $1\leq |E(C)\cap E(C')|\leq 2$. 

Note that  $v\in V^{**}$.
In the following, let $X=N_2[v]$ and $r=\gamma_{b,2}(G[X])$. Then $\beta \leq 4$.
By Lemma~\ref{Lem:no_intersection_induced cycle_length4}, $\beta=4$ and so $|X|=8$.  By Lemma~\ref{Lem:possible:cases}, $v$ is in one of the cases $(Q_{3,0,1})$ or $(Q_{4,1,0})$.
When $v$ is in the case $(Q_{3,0,1})$, we  denote by $w^*$ the isolated vertex in $G-X$. When $v$ is in the case $(Q_{4,1,0})$, we denote by $C^*$ the $C_4$-component in $G-X$.

Since $|E(C')|\in\{5,6\}$ and $|E(C)\cap E(C')|\in\{1,2\}$, there are four possible cases according to $|E(C')|$ and $|E(C)\cap E(C')|$
(see Figure~\ref{fig:Lem:no_inter_cycle_length4 and 5 6_ver2}).  We often use the following observations.
\begin{itemize}
\item[$(\dag)$] When $v$ is in the case $(Q_{3,0,1})$, by Lemma~\ref{Lem:no_301}, there is a vertex $u_i$ such that $N(u_i)\subset N(v)\cup N(w^*)$ with  $|N(v)\cap N(u_i)|\ge 2$. Such vertex must be $u_1$, and so the neighbor of $u_1$ not on $C$ is a  neighbor of $w^*$.
\item[$(\ddag)$] Suppose that $C'$ is a $5$-cycle. 
Then $\ell\ge 1$. Therefore, $v$ is not in the case $(Q_{4,1,0})$ by Lemma~\ref{Lem:possible:cases}.
Thus $v$ is in the case of $(Q_{3,0,1})$. 
Since $\beta_2+\ell\le 1$, $\beta_2=0$ and  $\ell=1$. 
\item[($\dag\dag$)] If $v$ is in the case  $(Q_{4,1,0})$, then $\beta_2=\ell=0$ and so 
every vertex in $B$ is a $3$-vertex and $B$ is an independent set.
\end{itemize}

We assume that the vertices  are labeled
as in the figures of Figure~\ref{fig:Lem:no_inter_cycle_length4 and 5 6_ver2}, in each of the corresponding case.
In the following, $x_1$ or $x_2$ is considered only when it appears on the figure. By Lemma~\ref{Lem:key3.1}(iii), $x_1$ is a $3$-vertex. 

\begin{figure}[h!]\centering
\includegraphics[page=10,width=18cm]{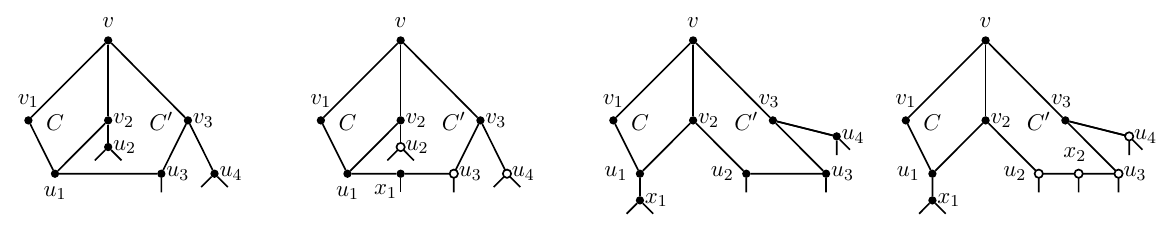}
 \caption{Subgraphs mentioned in Lemma~\ref{Lem:no_inter_cycle_length4 and 5 6}, first two are the cases where $|E(C)\cap E(C')|=2$, and the other two are the cases where $|E(C)\cap E(C')|=1$. 
} 
\label{fig:Lem:no_inter_cycle_length4 and 5 6_ver2}
\end{figure}

\noindent{(Case 1)} Suppose that $|E(C)\cap E(C')|=2$. See  the first and second figures of Figure~\ref{fig:Lem:no_inter_cycle_length4 and 5 6_ver2}. 
Note that 
from the choice of $C$ and $C'$ and the fact that $|X|=8$, all vertices with labeling in the figures are distinct.
Moreover, by the choice of $C$ and $C'$, 
$ N(u_2) \cap \{u_3,u_4\}=\emptyset$, 
and $x_1\not\in N(u_2)$.

Suppose that $v$ is in the case $(Q_{3,0,1})$. 
If $C'$ is a $6$-cycle, then
by $(\dag)$, $x_1$ is a neighbor of $w^*$, and so $N(w^*)\not\subseteq B$, a contradiction.
Thus $C'$ is a $5$-cycle. 
In addition,  
if $u_2$ and $u_3$ have a common neighbor $u'$, then
$u_1u_3u'u_2v_2u_1$ is a $5$-cycle sharing one edge with $C$, which contradicts the choice of $C$ and $C'$.
Thus
$u_2$ and $u_3$ have no common neighbor.
Thus, $w^*$ is not adjacent to $u_2$ or $u_3$.
Hence $d(w^*)=2$. By $(\dag)$, $u_3$ is a neighbor of $w^*$.
Thus $N(w^*)=\{u_3,u_4\}$. Then
$v_3u_4w^*u_3$ is a $4$-cycle sharing one edge with $C'$, which contradicts the choice of $C$ and $C'$.

Suppose that $v$ is in the case $(Q_{4,1,0})$.
Then $C'$ is a $6$-cycle by $(\ddag)$, and so see the second figure of Figure~\ref{fig:Lem:no_inter_cycle_length4 and 5 6_ver2}.
Note that ($\dag\dag$) holds. 
By Lemma~\ref{Lem:two-vertex-removal_C_4},
$u_2$, $u_3$, $u_4$ are adjacent to $C^*$.
Let 
$x_2$ and $x_3$ be the neighbors of $u_2$ other than $v_2$ and $Y=N_2[v_2]$.
Note $v_2\in V^t$.
Assume that $x_2$ is on $C^*$.
By the choice of $C$ and $C'$, $x_3x_1\notin E(G)$. 
If $x_2x_3\in E(G)$, then $x_3 $ in on $C^*$ and so, by Lemma~\ref{Lem:removal-two-edges} and the fact $\delta(G)\geq 2$, $i(Y)=a(Y)=0$, which contradicts Lemma~\ref{Lem:every}(i).
Thus $x_2x_3\notin E(G)$.
Then $u_2$ is the only neighbor of $x_3$ in $Y$ and  we check $v_1,x_3\in B(v_2)$.
By the fact $\delta(G)\geq2$ and Lemma~\ref{Lem:removal-two-edges}, $i(Y\setminus \{x_3\})=a(Y\setminus \{x_3\})=0$, which contradicts Lemma~\ref{Lem:every}(ii).

\noindent{(Case 2)}  Suppose that  $|E(C)\cap E(C')|=1$.
See  the third and fourth figures of Figure~\ref{fig:Lem:no_inter_cycle_length4 and 5 6_ver2}. 
Note that 
from the choice of $C$ and $C'$ and the fact that $|X|=8$ and $\ell\le 1$, all vertices with labeling in the figures are  distinct.

If $v$ is in the case $(Q_{3,0,1})$, then 
by ($\dag$), $x_1$ is a neighbor of $w^*$ and so $x_1\in B$, a contradiction. Hence $v$ is in the case $(Q_{4,1,0})$.
Thus ($\dag\dag$) holds.
By ($\ddag$),  $C'$ is a $6$-cycle.
By the choice of $C$ and $C'$,
\begin{equation}\label{eq:Lem:no_inter_cycle_length4 and 5 6_2N(x_1)}
N[x_1]\cap N[u_2]=\emptyset.  \end{equation}
If $x_2u_4\in E(G)$ or $d(x_2)=2$, then $x_2$ is an isolated vertex in $G-X$, which contradicts the case $(Q_{4,1,0})$.
Thus $x_2u_4\notin E(G)$ and $d(x_2)=3$.
Since $x_2$ is a $1$-vertex in $G-X$, $x_2\notin V(C^*)$.
By Lemma~\ref{Lem:two-vertex-removal_C_4},
at least three vertices in $B$ are adjacent to $C^*$. Thus both $u_i$ and $u_{i+1}$ are adjacent to $C^*$ for some $i\in\{1,3\}$.

Suppose that both $u_1$ and $u_2$ are adjacent to $C^*$.
Then $x_1$ is on $C^*$.
Let $x'_2$ be the neighbor of $u_2$ other than $v_2$ and $x_2$ with $x'_2\in V(C^*)$.
Then $x_1$ and $x'_2$ are not adjacent by \eqref{eq:Lem:no_inter_cycle_length4 and 5 6_2N(x_1)}.
We let $Y=N_2[v_2]\setminus \{x_1\}$.
We check that $v_1,x_1\in B(v_2)$ and $x_1$ has the unique neighbor $u_1$ in $N_2[v_2]$. 
Thus, by the fact $\delta(G)\geq 2$ and Lemma~\ref{Lem:removal-two-edges}, $a(Y)=i(Y)=0$, which contradicts Lemma~\ref{Lem:every}(ii).
Therefore both $u_3$ and $u_4$ are adjacent to $C^*$, and there is a vertex $x^*$ such that $x^*\in N(u_1)\cup N(u_2)$ and $x^* \notin V(C)\cup V(C')\cup V(C^*)$. 
Then $x^*\in B(v_2)$ and $x^*$ has exactly one neighbor that is $u_1$ or $u_2$ in $N_2[v_2]$.
Let $Y=N_2[v_2] \setminus \{x^*\}$. 
We can check
$a(Y)=0$ by Lemma~\ref{Lem:removal-two-edges}.
By Lemma~\ref{Lem:every}(ii), $i(Y)>0$.
For the isolated vertex $w'$ in $G-Y$, $N(w')=\{u_2,x_2\}$ or $\{u_1,x_2\}$, which is a contradiction to 
Lemma~\ref{Lem:key3.1}(ii) or \eqref{eq:Lem:no_inter_cycle_length4 and 5 6_2N(x_1)}.
\end{proof}

\begin{Lem}\label{lem:no-common-neighbors of degree$2$}
Suppose that $\beta(v)=4$ for $v\in V^{**}$. Let $N(v)=\{v_1,v_2,v_3\}$ and $v_1$ be a $2$-vertex.
Then
$N(v_1)\cap N(v_i)=\{v\}$ for each $i\in\{2,3\}$. Moreover, there is no $4$-cycle with a $2$-vertex in $G$.
\end{Lem} 
\begin{proof}
Note that the `moreover' part immediately follows from the former statement. Let $X=N_2[v]$.
To the contrary, suppose that  $(N(v_1)\cap N(v_2) )\setminus \{v\}\neq \emptyset$.
Since $\beta=4$, $N(v_1)\cap N(v_3)=\{v\}$. We follow the label of the vertices as in  Figure~\ref{fig:lem:no-common-neighbors of degree$2$}. 
We note that $d(u_1)=d(x_1)=3$ by Lemma~\ref{Lem:key3.1}(iii)
By Lemma~\ref{Lem:no_inter_cycle_length4 and 5 6},
Figure~\ref{fig:lem:no-common-neighbors of degree$2$} shows an induced structure, that is, 
\begin{equation}
  \label{eq:lem:no-common-neighbors of degree$2$_1} 
\ell=0\quad \text{and}\quad N[x_1]\cap \{u_2,u_3,u_4\}=\emptyset.\end{equation}
Since $\beta=4$,
it is one of the cases $(Q_{3,0,1})$ and $(Q_{4,1,0})$ by Lemma~\ref{Lem:possible:cases}. 

\begin{figure}[h!]\centering
\includegraphics[page=11,width=4.7cm]{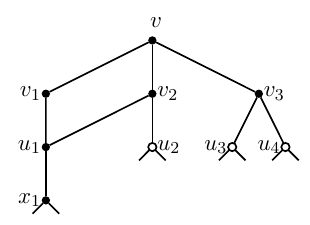}
 \caption{A subgraph mentioned in Lemma~\ref{lem:no-common-neighbors of degree$2$}}
\label{fig:lem:no-common-neighbors of degree$2$}
\end{figure}

Suppose that $v$ is in the case $(Q_{3,0,1})$.
Let $w$ be the isolated vertex in $G-X$. 
By Lemma~\ref{Lem:no_301}, there is a vertex $u\in B(v)\setminus N(w)$ such that $N(u)\subset N(w)\cup N(v)$ and $|N(u)\cap N(v)|\ge 2$ and so $u_1=u$. Thus $x_1 \in B(v)$, a contradiction. 

Suppose that $v$ is in the case $(Q_{4,1,0})$.
Then $\beta_2=\ell=0$ and so $d(u_2)=d(u_3)=d(u_4)=3$.
Let $C^*$ be the $C_4$-component in $G-X$.
By Lemma~\ref{Lem:two-vertex-removal_C_4}, at least three vertices of $u_1$, $u_2$, $u_3$, $u_4$ are adjacent to $C^*$.
By Lemma~\ref{Lem:no_inter_cycle_length4 and 5 6},
at least one of $u_1$ and $u_2$ is nonadjacent to $C^*$. 
Thus the set of vertices adjacent to $C^*$ is
either $\{u_1,u_3,u_4\}$ or $\{u_2,u_3,u_4\}$.

Suppose that 
the set of vertices adjacent to $C^*$ is
$\{u_2,u_3,u_4\}$.
We let $x_i$ be a neighbor of $u_i$ on $C^*$ for each $i\in \{2,3,4\}$.
Let $x'_2$ be the neighbor of $u_2$ other than $x_2$ and $v_2$.
By \eqref{eq:lem:no-common-neighbors of degree$2$_1},
$x_1\neq x_2'$.
By Lemma~\ref{Lem:no_inter_cycle_length4 and 5 6}, $x'_2x_1\notin E(G)$ and so $x'_2$ has exactly one neighbor that is $u_2$ in $N_2[v_2]$.
In addition, $v_1,x'_2\in B(v_2)$.
We let $Y'_2=N_2[v_2]\setminus \{x'_2\}$.
Then by the fact $\delta(G)\geq 2$ and Lemma~\ref{Lem:two-vertex-removal_C_4}, $a(Y'_2)=i(Y'_2)=0$, which contradicts Lemma~\ref{Lem:every}(ii).

Suppose that 
the set of vertices adjacent to $C^*$ is
$\{u_1,u_3,u_4\}$.
By Lemma~\ref{Lem:no_inter_cycle_length4 and 5 6},
$|\partial(C^*)|\ne 3$ and so $|\partial(C^*)|= 4$.
Without loss of generality,
$u_3$ is adjacent to two vertices on $C^*$.
Then we let $Z=N[v_3]\cup N[u_1]$.
By the fact that $\delta(G)\ge 2$ and Lemma~\ref{Lem:removal-two-edges}, $a(Z)=i(Z)=0$, which contradicts Lemma~\ref{Lem:two-sets}. 
\end{proof}

\begin{Lem}\label{Lem:$4$-cycle-(3,3,3)}
Let $v\in V^{**}$. Then $\beta(v)=5$. \end{Lem}
\begin{proof}
Let $N(v)=\{v_1,v_2,v_3\}$, $v_1$ be a $2$-vertex, and $X=N_2[v]$.  By Lemma~\ref{Lem:no_intersection_induced cycle_length4}, $\beta(v)\geq 4$.
To the contrary, suppose that $\beta(v)=4$.
By Lemma~\ref{lem:no-common-neighbors of degree$2$},   $N(v_2)\cap N(v_3)\neq \{v\}$. 
Then we assume that the vertices are labeled as in the first graph in Figure~\ref{fig:clm:$4$-cycle_(3,3,3)-vertex}.
 By Lemma~\ref{lem:no-common-neighbors of degree$2$} and Lemma~\ref{Lem:key3.1}(iii), 
 $d(u_1)=d(u_3)=d(x_1)=3$.
 By Lemma~\ref{Lem:possible:cases},
 $v$ is in the case $(Q_{3,0,1})$ or $(Q_{4,1,0})$.
 
\begin{figure}[h!]\centering
\includegraphics[page=12,width=0.9\textwidth]{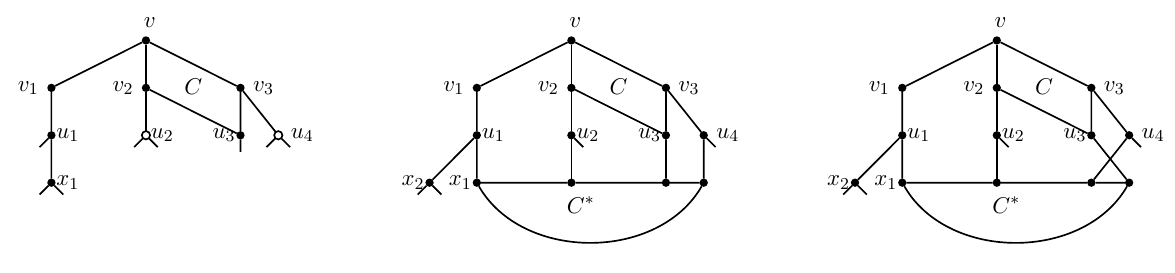}\vspace{-0.7cm}
  \caption{Subgraphs mentioned in Lemma~\ref{Lem:$4$-cycle-(3,3,3)}}
\label{fig:clm:$4$-cycle_(3,3,3)-vertex}
\end{figure}

{\noindent (Case 1)} Suppose that $v$ is in  the case $(Q_{3,0,1})$.
Then $\beta_2+\ell\le 1$.
 Let $w$ be the isolated vertex in $G-X$. 
By Lemma~\ref{Lem:no_301},
there is a vertex  $u\in B\setminus N(w)$  such that $N(u)\subset N(v)\cup N(w)$ and $|N(u)\cap N(v)|\ge 2$. 
Thus $u=u_3$.
Then $u_3 \notin N(w)$ and $u_3$ is adjacent to one of $u_1$, $u_4$, and $u_2$.
Suppose that $w$ is a $2$-vertex, then $wu_1\notin E(G)$ by Lemma~\ref{Lem:key3.1}(iii), which implies that $N(w)=\{u_2,u_4\}$. Since $u_2u_3\in E(G)$ or $u_3u_4\in E(G)$, it contradicts Lemma~\ref{Lem:key3.1}(ii).
Thus $w$ is a $3$-vertex. 
Then $wu_1 \in E(G)$ and so we may assume $w=x_1$.
Thus $x_1u_2 \in E(G)$ and $x_1u_4\in E(G)$.
If $u_3u_1\in E(G)$, then $N_2[u_3]=\mathcal{C}(X)$, which contradicts the case $(Q_{3,0,1})$.
Then $u_3u_1\notin E(G)$ and either $u_3u_2\in E(G)$ or $u_3u_4\in E(G)$.
Without loss of generality, we may assume $u_3u_2\in E(G)$.
Then $\ell=1$, $\beta_2=0$ and so $u_1u_4\notin E(G)$ and $d(u_4)=3$.
We let $Y_1=N_2[v_2]$.
Then we can check $\omega_G(Y_1)=22$, $\partial(Y_1)=4$, $d_{G-Y_1}(u_1)=d_{G-Y_2}(u_4)=1$ and so $a(Y_1)=i(Y_1)=0$, which contradicts Fact~\ref{Fact:main_inequality}.

{\noindent (Case 2)}  Suppose that $v$ is in the  case $(Q_{4,1,0})$. Then $\beta_2=\ell=0$ and so $d(u_2)=d(u_4)=3$.
 Let $C^*$ be the $C_4$-component in $G-X$. By Lemma~\ref{lem:no-common-neighbors of degree$2$}, $|\partial(C^*)|=4$.
By Lemma~\ref{Lem:two-vertex-removal_C_4}, there are at least three vertices that are adjacent to $C^*$.
By symmetry of $u_2$ and $u_4$, we may assume that
the number of vertices on $C^*$ adjacent to $u_4$ is greater than or equal to the number of those adjacent to $u_2$. Thus $u_4$ is adjacent to $C^*$.

First, we will show that $u_1$ is adjacent to exactly one vertex of $C^*$. Suppose to the contrary that that $u_1$ is not adjacent to $C^*$.
Then $u_2$ and $u_3$ are adjacent to exactly one vertex on $C^*$.
We note $u_3 \in V^*$ since $u_3$ is not on any triangle with $\beta(u_3)=5$. 
Let $X'=N_2[u_3]$ and $r'=\gamma_{b,2}(G[X'])$.
By the fact $\delta(G)\ge 2$, $G-X'$ has exactly one isolated vertex and $\ell(u_3)\ge 2$, which contradicts Lemma~\ref{Lem:possible:cases}.
Thus $u_1$ is adjacent to $C^*$. Without loss of generality, we may assume $x_1\in V(C^*)$. Let $x_2$ be the neighbor of $u_1$ other than $v_1$ and $x_1$.
If $x_2\in V(C^*)$, then by Lemma~\ref{Lem:key3.1}(ii), $u_1$ is not on a triangle and $x_1$ and $x_2$ are not adjacent and we check $u_1\in V^{**}$   and $\beta(u_1)=3$, which contradicts Lemma~\ref{Lem:no_intersection_induced cycle_length4}.
Thus $x_2\notin V(C^*)$, and so $u_1$ is adjacent to exactly one vertex of $C^*$.

We will show that $C^*$ is adjacent each of $u_1$, $u_2$, $u_3$, and $u_4$.
Suppose not.
Then by assumption,  $u_4$ is adjacent to two vertices of $C^*$.
By the fact $\delta(G)\ge 2$ and Lemma~\ref{Lem:removal-two-edges},
 $a(N_2[u_4])=0$  and $i(N_2[u_4])=0$.
  We can check $w_{G}( N_2[u_4])=24$ and $\partial(N_2[u_4])\le5$, which contradicts Fact~\ref{Fact:main_inequality}. Thus, $C^*$ is adjacent to each vertex of $B$.

Without loss of generality,
we may assume
the distance between $u_2$ and $x_1$ is two.
See the second and third figures of Figure~\ref{fig:clm:$4$-cycle_(3,3,3)-vertex}. 
Let $x_4$ be the neighbor of $u_4$ not in $V(C)\cup V(C^*)$.
Then $x_4\in B(v_3)$.
We can check $v_1\in B(v_3)$, $a(N_2[v_3]\setminus \{x_4\})=i(N_2[v_3]\setminus\{x_4\})=0$, and $u_4$ is the only neighbor of $x_4$ in $N_2[v_3]$, which contradicts Lemma~\ref{Lem:every}(ii).
Hence, we have shown $\beta(v)=5$. 
\end{proof}

\begin{Cor}\label{cor:4cycle:333}
Every vertex on a $4$-cycle is a $(3,3,3)$-vertex.
\end{Cor}
\begin{proof}
For a $4$-cycle $C$ in $G$, if a vertex $v$ on $C$ is not a $(3,3,3)$-vertex, then $v$ is a $(2,3,3)$-vertex by Lemma~\ref{Lem:key3.1}(iii) and Lemma~\ref{lem:no-common-neighbors of degree$2$} and so $\beta(v)\leq 4$, which is a contradiction to Lemma~\ref{Lem:$4$-cycle-(3,3,3)}.
\end{proof}

\subsection{Proof of Theorem~\ref{Thm:Main_subcubic}}
\label{subsect:3.4}
For a vertex $v \in V^{**}$, let $N(v)=\{v_1,v_2,v_3\}$, and assume that $v_1$ is a $2$-vertex. By Lemma~\ref{Lem:$4$-cycle-(3,3,3)}, 
 $|N_2[v]|=9$ and so the vertices in Figure~\ref{fig:nine-vertices} are distinct.  

\begin{figure}[h!]\centering
\includegraphics[page=13, width=5.2cm]{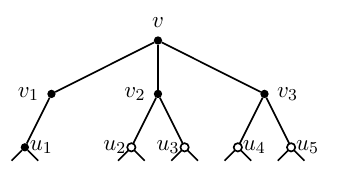}
\caption{A subgraph mentioned in the proof of Lemma~\ref{Lem:last} and Theorem~\ref{Thm:Main_subcubic}}
\label{fig:nine-vertices} 
\end{figure}

\begin{Lem}\label{Lem:last}
For a vertex $v \in V^{**}$,
let $X=N_2[v]$ and $u_1$ be the neighbor of the $2$-neighbor of $v$ other than $v$.
Then $v$ is in the case $(Q_{4,1,0})$ in Lemma~\ref{Lem:possible:cases} and $u_1$ is adjacent to exactly one vertex of the $C_4$-component of $G-X$.
\end{Lem}
\begin{proof} 
Let $N(v)=\{v_1,v_2,v_3\}$. 
We follow the vertex labeling in Figure~\ref{fig:nine-vertices}. 

Suppose that $v$ is not in the case $(Q_{4,1,0})$.
If $v$ is in the case $(Q_{3,0,1})$,
then by Lemma~\ref{Lem:no_301}, there exists a vertex $u$ such that $|N(u)\cap N(v)|\ge 2$, which contradicts the fact $\beta=5$.
Thus $v$ is in the case $(Q_{4,0,2})$ by Lemma~\ref{Lem:possible:cases}.

Let $w$ and $w'$ be the isolated vertices in $G-X$.
We note that $w$ and $w'$ are $2$-vertices
by Lemma~\ref{Lem:possible:cases}.
Then by Corollary~\ref{cor:4cycle:333} and Lemma~\ref{Lem:key3.1}(iii),
$d(u_i)=3$ for each $i\in \{2,3,4,5\}$ and 
$G[\mathcal{C}(X)]$ has a spanning subgraph isomorphic to the one in Figure~\ref{fig:nine-vertices_2}. Thus every vertex in $B$ is not a $(3,3,3)$-vertex and so by Corollary~\ref{cor:4cycle:333}, there is no $4$-cycle containing a vertex in $B$.

\begin{figure}[h!]\centering
\includegraphics[page=14,width=5cm]{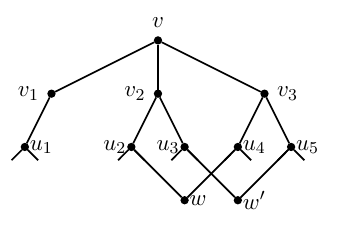}
 \caption{A subgraph mentioned in the proof of Lemma~\ref{Lem:last}}
\label{fig:nine-vertices_2}
\end{figure}

Since $v_1\in N[v]$, $Z=N[v]\cup N[w]$ contains at least one $2$-vertex other than $v$ and $w$.
In addition, by Lemma~\ref{Lem:$4$-cycle-(3,3,3)} and the fact $\delta(G)\ge 2$, $|E(N[v],N[w])|\ge 2$ and so by Lemma~\ref{Lem:two-sets},
$a(Z)>0$ or $i(Z)>0$.
Since there is no $4$-cycle containing a vertex of $B$, $a(Z)=0$. Thus, $i(Z)>0$.
Let $w^*$ be an isolated vertex of $G-Z$. Then by the fact $I(X)=\{w,w'\}$, $w^*=u_1$.
Since $u_1$ is a $3$-vertex in $G$,
$u_1$ is adjacent to $u_2$ and $u_4$.
Thus $u_1u_2wu_4u_1$ is a $4$-cycle and $u_1$ is not a $(3,3,3)$-vertex, which contradicts Corollary~\ref{cor:4cycle:333}.
Hence $v$ is in the case $(Q_{4,1,0})$. 

\begin{figure}[h!]\centering
\includegraphics[page=15, width=12cm]{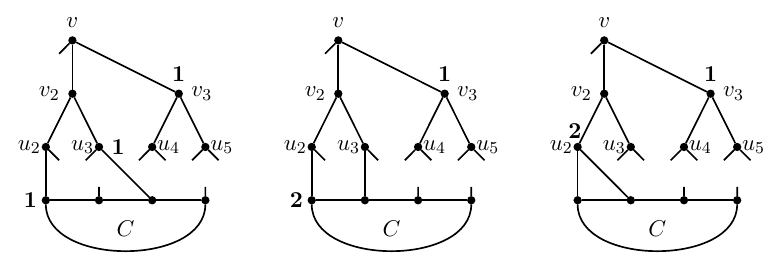}\vspace{-0.5cm}
  \caption{Subgraphs mentioned in the proof of Lemma~\ref{Lem:last} }
\label{fig:nine-vertices_final}
\end{figure}

To complete the proof, we let $C$ be the $C_4$-component in $G-X$. Note that  $|\partial(C)|=4$ 
by Lemma~\ref{lem:no-common-neighbors of degree$2$}. 
In addition, $u_1$ is adjacent to at most one vertex on $C$ by Lemmas~\ref{Lem:key3.1}(ii) and~\ref{Lem:$4$-cycle-(3,3,3)}. It remains to show that  $u_1$ is adjacent to $C$. Suppose that $u_1$ is not adjacent to $C$. Let $Z=\mathcal{C}(X)\setminus\{v_1,u_1\}$. 
Then $\omega_G(Z)\ge 33$ and, by Lemma~\ref{Lem:$4$-cycle-(3,3,3)}, $|\partial(Z)|\le 5$ and 
$G[Z]$ has a spanning subgraph isomorphic to one graph in Figure~\ref{fig:nine-vertices_final}.
Therefore $\gamma_{b,2}(G[Z])\le 3$.
By Fact~\ref{Fact:main_inequality}, $G-Z$ has an isolated vertex or a $C_4$-component.
Then an isolated vertex or a $C_4$-component is also in $\mathcal{C}(X)$, a contradiction.
\end{proof}

 Now we are ready to complete the proof of Theorem~\ref{Thm:Main_subcubic}.
\begin{proof}[Proof of Theorem~\ref{Thm:Main_subcubic}]
By Lemmas~\ref{Lem:existence_deg_1_2} and~\ref{Lem:key3.1}(iii),
$G$ has a $3$-vertex $v$ that has a $2$-neighbor $v_1$. 
Then $v$ is not on any triangle by Lemma~\ref{Lem:key3.1}(ii) and so $v\in V^{**}$.
We follow the vertex labeling in Figure~\ref{fig:nine-vertices}. 
By Lemma~\ref{Lem:last}, $v$ is in the case $(Q_{4,1,0})$ and $G-X$ has exactly one $C_4$-component $C$ that is adjacent to $u_1$ and $u_1$ is adjacent to one vertex on $C$. Since $v$ is in the case $(Q_{4,1,0})$, each vertex in $B$ is a $3$-vertex and $E(G[B])=\emptyset$ by Lemma~\ref{Lem:possible:cases}.

We also let $X'=N_2[u_1]$.
We note $u_1 \in V^{**}$.
By Lemma~\ref{Lem:last} again, $G-X'$  has exactly one $C_4$-component $C'$ that is adjacent to $v$ and $v$ is adjacent to one vertex on $C'$. We may assume 
$V(C')=\{u_2,v_2,u_3,w\}$ for some vertex $w$.
Clearly, $V(C)\cap V(C')=\emptyset$.
See the first figure of Figure~\ref{fig:nine-vertices_final3}.

\begin{figure}[h!]\centering
\includegraphics[page=16,width=16cm]{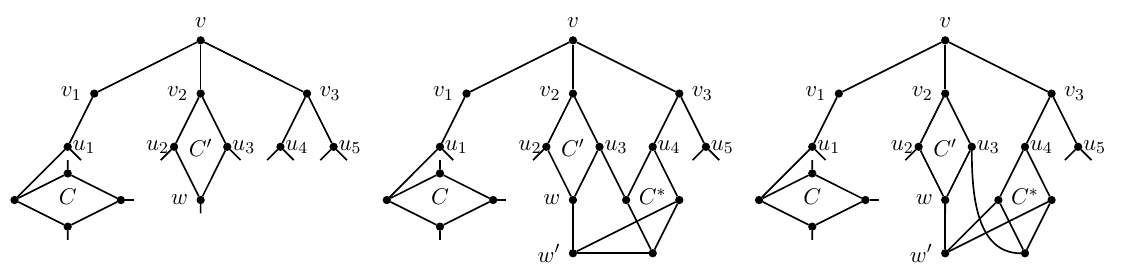}
\caption{Subgraphs mentioned in the proof of Theorem~\ref{Thm:Main_subcubic}}
\label{fig:nine-vertices_final3}
\end{figure}

By Corollary~\ref{cor:4cycle:333}, $d(w)=d(w')=3$ where $w'$ is the neighbor of $w$ other than $u_2$ and $u_3$.
Since $G-X$ has no isolated vertex, $w'\notin \{u_1,u_4,u_5\}$.
Let $Z= N[v] \cup N[w]$.
By Lemma~\ref{Lem:two-sets},
$a(Z)>0$ or $i(Z)>0$. 
We can check that $u_1$, $u_4$, and $u_5$ are not isolated vertices in $G-Z$.
If $i(Z)>0$, then there exists a vertex $w''$ such that $w''\neq w$ and $N(w'')=\{u_2,u_3,w'\}$ by Corollary~\ref{cor:4cycle:333} and so $C'$ cannot be the $C_4$-component in $G-X'$, which is impossible.
Hence $i(Z)=0$. Thus $a(Z)>0$, that is, $G-Z$ contains a $C_4$-component $C^*$. Note that $u_1\not\in V(C^*)$.  Thus the set of vertices adjacent to $C^*$ are contained in $\{u_2,u_3,v_3,w'\}$. 

If $V(C^*)\cap \{u_4,u_5\}=\emptyset$, then by Corollary~\ref{cor:4cycle:333},
$w'$ is adjacent to two vertices on $C^*$ and both $u_2$ and $u_3$ is adjacent to $C^*$ and so $\scalebox{1.4}{$\omega$}(G[ N[w] \cup V(C^*)])=26\equiv 8 \pmod{9}$ and $|\partial( N[w] \cup V(C^*))|=2$, which contradicts Lemma~\ref{Lem:cutedge4}.
Then $V(C^*)\cap \{u_4,u_5\}\neq \emptyset$.
Since $|E[V(C),\{u_2,u_3,u_4,u_5\} ]|=3$ by Lemma~\ref{Lem:last},  it holds that $|V(C^*)\cap \{u_4,u_5\}|=1$, 
$w'$ is adjacent to two vertices of $C^*$,
and exactly one of $u_2$ and $u_3$ is adjacent to $C^*$.
Without loss of generality, we may assume $u_4 \in V(C^*)$ and  $u_3$ is adjacent to $C^*$.
See   the two last figures of Figure~\ref{fig:nine-vertices_final3}.
We let $Z^*= N[w]\cup V(C^*)$.
Then one can check $\omega_G(Z^*)=24$, $|\partial(Z^*)|=4$, $\gamma_{b,2}(Z^*)\leq 2$ by assigning $2$ to $w'$, $a(Z^*)=i(Z^*)=0$, which contradicts Fact~\ref{Fact:main_inequality}.
\end{proof}

\bibliographystyle{plain}

\end{document}